\documentclass[12pt,a4paper]{article}

\usepackage{bookmark}
\usepackage[dvipsnames]{xcolor}
\usepackage[english]{babel}
\usepackage{amssymb}
\usepackage{amsmath,mathrsfs}
\usepackage{indentfirst}
\usepackage{epsfig}
\usepackage{tikz}
\usetikzlibrary {matrix} 
\usepackage{tikz}
\usetikzlibrary{shadings} 

\usepackage{graphicx}
\usepackage{caption}
\usepackage{subcaption}
\usepackage{tikz} 
\usetikzlibrary{fit,positioning}
\usetikzlibrary{arrows,positioning,shapes,backgrounds,fit}

\def\cambD{\marginpar{\textcolor{blue}{Dung changed}}\textcolor{blue}}

\usepackage{amsmath, amsthm, amssymb, amsfonts} 
\usepackage{relsize}
\usepackage{geometry}
\usepackage{mathrsfs}
\usepackage{mathtools}
\usepackage{tikz}
\usepackage{tikz-cd}
\usepackage{listings}
\usepackage{extarrows}
\usepackage{cancel} 
\usepackage{thmtools}
\usepackage{thm-restate}
\usepackage[inline]{enumitem}
\usepackage{indentfirst}
\usepackage{cleveref}
\usepackage{empheq}
\usepackage{diagbox}
\usepackage{tcolorbox}
\usepackage{hyperref}

\hypersetup{
    colorlinks=true,       
    linkcolor=blue,          
    citecolor=green,        
    filecolor=cyan,         
    urlcolor=magenta        
}

\newcommand{\R}{\mathbb{R}}
\newcommand{\C}{\mathbb{C}}

\newcommand\norm[1]{\left\lVert#1\right\rVert}

\theoremstyle{definition}

\newtheorem{theorem}{Theorem}[section]

\newtheorem{definition}{Definition}[section]
\newtheorem{remark}{Remark}[section]

\theoremstyle{definition}

\DeclareMathOperator{\tr}{tr}

\mathchardef\mhyphen="2D

\begin{document} 

\title{Existence of Quantum Splines via Fourth-Order Gradient Flows}

\author{
{\sc Chun-Chi Lin} 
\thanks{Department of Mathematics, National Taiwan Normal University, Taipei, 116 Taiwan, \url{chunlin@math.ntnu.edu.tw}}, 
~ 
{\sc Yang-Kai Lue} 
\thanks{Department of Applied Mathematics, National Yang Ming Chiao Tung University, Taiwan, 
\url{luf961@gmail.com}}, 
~ and ~
{\sc Dung The Tran} 
\thanks{
VNU University of Science, \url{tranthedung56@gmail.com}} 
} 


\date{February 26, 2026}

\maketitle



\begin{abstract} 

We establish a rigorous existence theory for the quantum splines introduced by Brody, Holm, and Meier in Physical Review Letters (2012). These curves arise as solutions of a variational problem on the unitary group describing optimally controlled quantum evolutions. By formulating the problem within a geometric gradient–flow framework for Riemannian spline interpolation, we construct a well-posed fourth-order evolution whose asymptotic limits realize the desired quantum splines. The analysis requires adapting the variational structure to boundary conditions dictated by the physical model, which are not directly amenable to the setting in our recently developed framework for gradient flows of Riemannian spline interpolation. We show that, despite these difficulties, the modified system admits a rigorous analytical treatment, yielding both existence and a constructive procedure for generating quantum splines. Our results provide a mathematical foundation for the variational description of smooth quantum control trajectories and clarify the analytical structure underlying their formation. 

\end{abstract}

\tableofcontents

\baselineskip=15 pt plus .5pt minus .5pt

\section{
Introduction
}


Quantum splines, introduced by Brody, Holm, and Meier \cite{BHM12}, are smooth curves in the unitary group 
\[
\mathbf{SU}(N):=\{U\in M_N(\mathbb{C}): \det U=1, U U^\ast =\operatorname{Id}_{N\times N} \}
\] 
characterized by a variational principle. These curves minimize the trace norm of the time derivative of the associated Hamiltonian while ensuring that the induced trajectories in the state space pass either exactly or approximately through a prescribed sequence of target states at specified times. 
In other words, given an initial
state $\psi_{0}\in \mathbf{SU}(N)$ at physical time $x_0$, a set of target states $\phi_{1},\ldots, \phi_{q}$ in $\mathbf{SU}(N)$ at times $x_1,\ldots,x_q$, and an initial Hamiltonian $H_0$ at physical time $x_0$, a map 
$U: I=[x_0, x_q]\rightarrow \mathbf{SU}(N)\subset \mathbf{GL}(N;\mathbb{C})$ representing a curve in $\mathbf{SU}(N)$ 
is said to be a quantum spline if it satisfies the Schr\"{o}dinger equation 
\begin{equation}
\label{eq:Schoedinger_Eq}
\partial_x U=-\mathrm{i} H U  
\end{equation}
and is an equilibrium configuration of the functional 
\begin{equation} 
\label{eq:Action_J} 
\mathcal{J}_{QS}(H)=\frac{1}{2}\int_{I} \, 
\left\langle \mathrm{i} \partial_x H, \mathrm{i} \partial_x H \right\rangle 
\, dx 
+ \frac{1}{2\sigma^2} \sum_{j=1}^{q} \text{dist}^2_{\mathbf{SU}(N)}(U(x_j)\psi_0, \phi_j) 
\end{equation} 
with the boundary conditions $U(x=0)=U_0$ ($U_0$ may be chosen to be the identity) and  
$\partial_x U(x=0)=-\mathrm{i} H_0 U_0$. 
Here, $\sigma>0$ is a penalization parameter, typically chosen to be small, 
$x\in I=[x_0, x_q]$ denotes the physical time variable, 
and $\mathrm{i} H\in \mathfrak{su}(N)$ represents 
the Lie algebra, defined by 
\begin{equation*}
\mathfrak{su}(N):= 
\{X\in M_N(\mathbb{C}): X^\ast=-X, \tr{X}=0  \}.   
\end{equation*}

We recall that the inner product on 
$\mathbf{M}_{N}(\C)$ used in 
\eqref{eq:Action_J} is defined by
\begin{equation}
\label{eq:Matrix_inner_product}
\left\langle A,B \right\rangle = \tr(B^\ast A), 
\end{equation} 
where $B^\ast$ is the conjugate transpose of $B$. 
The associated matrix norm is 
\begin{align}
\label{def:norm-matrix}
|B| := (\tr(B^\ast B))^{1/2}.  
\end{align}  
Note that $M_N(\mathbb{C})$ can be viewed as the real Euclidean space $\mathbb{R}^{2N^2}$ equipped with the inner product $\langle\cdot,\cdot\rangle$ given in \eqref{eq:Matrix_inner_product}. 
Under this identification, each matrix $A=(a_{jk})$ is represented by the
$2N^2$ real components 
$(\operatorname{Re} a_{jk},\,\operatorname{Im} a_{jk})$, 
where $j,k=1,\cdots,N$, so the inclusion
$\mathbf{SU}(N)\subset M_N(\mathbb{C})$ becomes the standard embedding 
\[
\iota:\mathbf{SU}(N)\hookrightarrow\mathbb{R}^{2N^2},\qquad \iota(A)=A.
\]
This map is an isometric immersion, as $\mathbf{SU}(N)$ 
is equipped with the bi-invariant metric induced from the inner product $\langle\cdot,\cdot\rangle$. 

To enforce the pointwise differential constraint in \eqref{eq:Schoedinger_Eq}, Brody et al. in \cite{BHM12} utilized the method of Lagrange multipliers, introducing the augmented functional:
\begin{equation} 
\label{eq:Action_J_Q}
\mathcal{J}(H)
=\frac{1}{2}\int_{I} \, 
\left(
\left\langle \mathrm{i} \partial_x H, \mathrm{i} \partial_x H \right\rangle 
+ \left\langle M, 
\mathrm{i} H+\partial_x U \cdot U^{-1} \right\rangle \right) \, dx 
+ \frac{1}{2\sigma^2} \sum_{j=1}^{q} \text{dist}^2_{\mathbf{SU}(N)}(U(x_j)\psi_0, \phi_j).    
\end{equation} 
Here, $M=M(x)$ serves as the Lagrange multiplier enforcing the constraint 
$\mathrm{i} H+\partial_x U \cdot U^{-1}=0$, 
which arises directly from the Schrödinger equation. 
This augmented formulation replaces the direct minimization of $\mathcal{J}_{QS}$ in \eqref{eq:Action_J}, while ensuring that the pointwise constraint dictated by \eqref{eq:Schoedinger_Eq} is satisfied. 
The formulation of quantum splines, which are characterized as equilibrium configurations of a second-order system, 
permits the use of analytical tools from second-order elliptic PDEs and control theory, such as the Pontryagin Maximum Principle.

In this article, we choose a different formulation and approach for studying the quantum splines in \cite{BHM12}. 
Our main contribution is to establish an existence theory for the quantum splines of Brody, Holm, and Meier in \cite{BHM12}, via our recently developed framework for treating the $2k$-th order gradient flows of Riemannian spline interpolation and up to a modification of the boundary conditions, together with a rigorous and constructive procedure for obtaining such quantum splines as asymptotic limits of the associated gradient flow. 
Although the overall argument follows a similar structure to our framework for gradient flows of Riemannian splines, the nonstandard boundary conditions arising from the variational setting of Brody, Holm, and Meier introduce non-trivial complexities in applying Solonnikov's theory to the parabolic system of partial differential equations, in particular in verifying the complementary conditions. 
Specifically, in order to obtain a well-posed gradient flow, it is necessary to impose additional boundary conditions beyond those considered in \cite{BHM12} at the end point $x_q$.
One natural choice is to impose the second-order condition  
$$
D_x\partial_x U_q(t,x_q)=0, 
$$ 
since any equilibrium configuration without prescribed boundary data at $x_q$ must satisfy this relation. This follows directly from (the boundary terms of) the first variation formula applied to the functional 
\begin{align} 
\label{eq:J_2=int_|H_x|^2}
\mathcal{I}(H):=\frac{1}{2}\int_{I} \, \left\langle \mathrm{i} \partial_x H, \mathrm{i} \partial_x H \right\rangle \, dx. 
\end{align} 
This boundary condition is consistent with the variational formulation of quantum splines in \cite{BHM12}.  
Another option is to impose a Dirichlet boundary condition at the endpoint $x_q$, i.e., further impose 
$$ 
\partial_x U_q(t, x_q)=\varphi^\prime_q, 
$$ 
which is analogous to the condition prescribed at $x_0$. 
The second choice provides an alternative variational setting for quantum splines when the terminal quantum state $\varphi^\prime_q$ 
at the endpoint (or final time) is specified. 
These conditions necessitate significant refinements to the analysis in order to apply Solonnikov's theory to obtain the existence and construction results proved here. 
In contrast to the approach and results in \cite{BHM12}, where only a numerical scheme was proposed, 
our method specializes in the case $k=2$ with the modification of the boundary conditions considered in \cite{LT26} and establishes a constructive characterization of quantum spline interpolation. 
As one of the results, any asymptotic limit of global solutions to the fourth-order gradient flow studied here defines a quantum spline $U: [x_0, x_q] \rightarrow \mathbf{SU}(N)$, which satisfies the least-squares fitting property and the boundary conditions $U(x_0)=U_0$ and $\partial_x U(x_0)=\varphi^\prime_0:=-\mathrm{i} H_0 U_0$ within the class $C^{2}([x_0,x_q];\mathbf{SU}(N))$.


Although Brody, Holm, and Meier \cite{BHM12} have indicated that the quantum splines they defined are Riemannian cubics in $\mathbf{SU}(N)$, we include a detailed derivation for completeness in the following. 
Throughout the remainder of this article, we use the notation 
$M_x$, $M_{xx}$, etc., to denote the partial derivatives of the matrix-valued map $M$, that is, $\partial_x M$, $\partial_x^2 M$, and so on.
In \cite{BHM12}, the Hamiltonian for the quantum splines is given by \eqref{eq:J_2=int_|H_x|^2}. 
For $U: I=[x_0, x_q]\to \mathbf{SU}(N)$   
subject to \eqref{eq:Schoedinger_Eq}, 
the following identity holds:  
\begin{align*}  
\mathcal{I}(H) 
= \frac{1}{2}\int_{I} \, \left\langle (U_x\cdot U^{-1})_x, (U_x\cdot U^{-1})_x \right\rangle \, dx. 
\end{align*}

Below, we demonstrate that the $H$ in the 
Schr\"{o}dinger equation \eqref{eq:Schoedinger_Eq} 
is Hermitian and traceless:  
By differentiating the identity 
$U U^{\ast}=\operatorname{Id}_{N \times N}$, 
with respect to $x$, we obtain     
\begin{align}\label{eq:partial_x(U*)}
U_x^\ast = - U^\ast U_x U^\ast 
\text{.}
\end{align}
Thus,  
\begin{align}\label{H-Hermitian}
H^{\ast}=(\mathrm{i} U_x U^{\ast})^{\ast}=-\mathrm{i} U U^{\ast}_x=\mathrm{i} U_x U^{\ast}=H
\text{.} 
\end{align} 
By differentiating 
the constraint $\det U^{\ast} = 1$ 
with respect to $x$, and applying Jacobi’s formula 
$(\det Y)_x = \det Y \cdot \tr(Y^{-1} Y_x)$  
for a matrix $Y$, 
together with the identity $U^\ast=U^{-1}$ and the relation \eqref{eq:partial_x(U*)}, we obtain  
\[
0 = (\det U^{\ast})_x 
=\det U^{\ast} \cdot \tr( (U^{\ast})^{-1} U_x^{\ast})
=\tr(U U_x^{\ast})
= - \tr(U_x U^\ast)
= -\tr(-\mathrm{i} H).
\] 
These identities imply that  
\begin{align}\label{trace-H=0} 
\tr(H)=0.  
\end{align} 
From \eqref{H-Hermitian} and \eqref{trace-H=0}, we infer that $H$ is Hermitian and traceless. 
Consequently, 
\begin{equation}
\label{eq:iH=}
\mathrm{i}H=-U_x U^\ast \in \mathfrak{su}(N).
\end{equation}

Let $D$ denote the Levi–Civita connection on $\mathbf{SU}(N)$. Since this bi-invariant Riemannian metric can also be induced from the ambient Euclidean space, the connection $D$ can be expressed as the orthogonal projection of the ambient derivative onto the tangent bundle 
\[ 
T\mathbf{SU}(N) 
= \{(U, XU) \mid X \in \mathfrak{su}(N) 
\}. 
\] 
Note that for any matrix $AU$ with 
$A \in \mathbf{GL}(\mathbb{C},N)$ and 
$U \in \mathbf{SU}(N)$ 
the orthogonal projection onto the tangent space of $\mathbf{SU}(N)$ is expressed as 
\begin{align}\label{def:projection}
P(AU) = \operatorname{skew}_0(A)\, U,
\end{align} 
where 
\begin{align}\label{def:skew-operator}
\operatorname{skew}(A) = \frac{A - A^\ast}{2}, \qquad
\operatorname{skew}_0(A) 
= \operatorname{skew}(A)
- \frac{\tr(\operatorname{skew}(A))}{N}\, \operatorname{Id}_{N \times N} .
\end{align} 
Hence, we use 
\begin{align}
\label{D_xY=projection}
D_x Y=(\partial_x Y)^{T}=P(\partial_x Y), 
~~~~\forall\, Y\in T\mathbf{SU}(N) 
\end{align} 
to compute the covariant derivatives. 
Notice that, by letting $Y=U_x$ in \eqref{D_xY=projection} and applying the 
Schr\"{o}dinger equation \eqref{eq:Schoedinger_Eq}, we have 
\[
D_x U_x 
=P(-\mathrm{i}H_x U) +P(-H^2 U) 
\text{.}
\]
From \eqref{H-Hermitian}, \eqref{trace-H=0}, and \eqref{eq:iH=}, 
we have 
\begin{empheq}[left=\empheqlbrace\quad]{align}
\label{eq:skew(H)} 
\begin{split}
&\operatorname{skew}(\mathrm{i} \, \partial_x^\ell H)
= -\mathrm{i} \, \partial_x^\ell H
=\operatorname{skew}_{0}(\mathrm{i} \, \partial_x^\ell H), 
~~~ \forall\, \ell\in\mathbb{N}_0, 
\\ 
&\tr( \partial_x^\ell H) = 0, 
~~~ \forall\, \ell\in\mathbb{N}_0, 
\\ 
&\operatorname{skew}(H^2) 
=\frac{H^2 - (H^2)^\ast}{2}
= \frac{H^2 - H^2}{2} = 0. 
\end{split} 
\end{empheq} 
Therefore, from \eqref{def:projection} and \eqref{def:skew-operator}, we obtain the decomposition of the tangential and normal components, i.e.,  
\begin{empheq}[left=\empheqlbrace\quad]{align} 
\label{eq:d_x(U_x)=}
\begin{split}
\left(\partial_x U_x \right)^{T}
&= D_x U_x 
=P(\partial_x U_x) 
= -\mathrm{i} H_x U, 
\\ 
\left(\partial_x U_x \right)^{\perp}
&= \partial_x U_x - \left(\partial_x U_x \right)^{T}
= - H^2 U 
\text{.}
\end{split}
\end{empheq} 
Consequently, we obtain 
\begin{align}
\label{eq:E}
\mathcal{I}(H)=
\frac{1}{2} \int_I |D_x U_x|^2 \, dx
=: \mathcal{F}(U), 
~~~~ \text{ where }~~ 
H=\mathrm{i} U_x U^\ast, 
~ U\in \mathbf{SU}(N). 
\end{align} 
It follows from \eqref{eq:E} that the quantum splines defined in \cite{BHM12} coincide with the Riemannian cubics in $\mathbf{SU}(N)$ studied in the literature.
There is an extensive body of work on Riemannian cubics and, more generally, higher-order Riemannian splines on Riemannian manifolds; 
we mention only a few representative references here, namely \cite{NHP89, GGP03, N03, GHMRV12-1, GHMRV12-2, MTM17, HRW19, CSC22, CH24}, listed in chronological order, and refer the interested reader to these works for further details.  
In this article, $\mathbf{SU}(N)$ is regarded as a manifold equipped with its Lie group structure and a bi-invariant Riemannian metric.
Background material on covariant differentiation, Lie brackets, and the Riemannian curvature tensor in this setting can be found in 
\cite{Milnor63, Nomizu54}; see also Lemma 2 and Proposition 1 of \cite{BCC21} for a more direct formulation.

\section{Quantum Spline Interpolation via Gradient Flows}
\label{sec:Spline_Interpolation+Flow}

In this section, we apply the mathematical framework recently developed in \cite{LT26} to provide a constructive approach to establishing the existence of the quantum splines introduced in \cite{BHM12}.  
We note that there is a difference in the variational settings of \cite{BHM12} and \cite{LT26} regarding the boundary conditions at the endpoint $x_q$. 
Namely, in \cite{BHM12} no boundary conditions were prescribed at $x_q$ (denoted by $t_m$ therein), which implies that the induced gradient flow is under-determined.  
Our method in \cite{LT26} employs a fourth-order geometric heat flow approach, which not only provides a constructive method for obtaining solutions but also strengthens the practical feasibility of computing quantum splines. 
The prescribed points correspond to quantum circuits in the context of quantum computing. 
The quantum splines constructed in this work thus offer a Riemannian optimization framework over the special unitary group $\mathbf{SU}(N)$, exhibiting desirable optimization properties. 
The asymptotic limits of the global solutions to the gradient flow studied in this article satisfy those of quantum splines in \cite{BHM12}, although the notation and expressions are slightly different.

\vspace{2.5cm}  

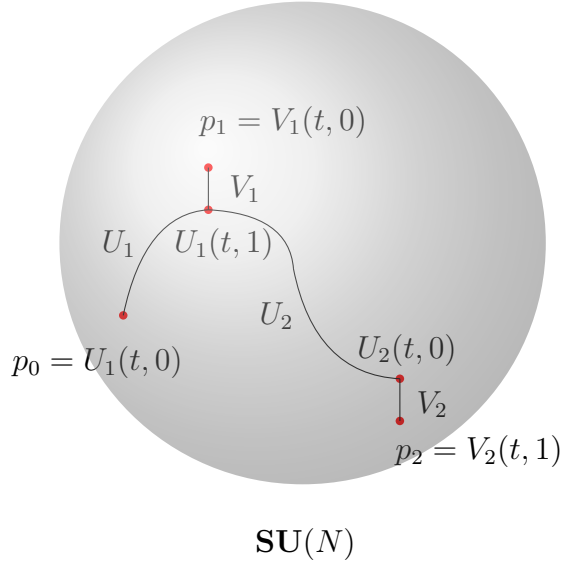
\begin{figure}[h]
\setlength{\unitlength}{0.56 mm}
\begin{center}
\begin{picture}(120, 75)

{\color{red}
\put(15,40){\circle*{2}} 
\put(35,65){\circle*{2}} 
\put(80,25){\circle*{2}} 
\put(35,75){\circle*{2}} 
\put(80,15){\circle*{2}} 
}

\put(-11,27){$p_0=U_1(t,0)$}
\put(27,55){$U_1(t,1)$}
\put(70,30){$U_2(t,0)$}
\put(33,84){$p_1=V_{1}(t,0)$}
\put(79,06){$p_2=V_2(t,1)$}
\put(10,55){$U_1$}
\put(47,38){$U_2$}
\put(40,68){$V_1$}
\put(84,17){$V_2$}
\put(46,-16){$\mathbf{SU}(N)$ }

\qbezier(15,40)(20,64)(35,65)
\qbezier(35,65)(35,67)(35,75)
\qbezier(35,65)(54,64)(55,51)
\qbezier(55,51)(60,26)(80,25)
\qbezier(80,25)(80,22)(80,15)



\begin{tikzpicture}
  \shade[ball color=gray!40, opacity=0.4] (0,0) circle (3.2cm);
\end{tikzpicture}





\end{picture}
\vspace{1.0cm}
\caption{The case of $q=2$.} 
\end{center}
\end{figure}

Denote by $\mathbb{N}=\{1,2,\ldots \}$ the set of natural numbers and by $\mathbb{N}_0=\{0,1,2,\ldots \}$ the set of non-negative integers. 
Let 
\[
\mathcal{P}=\{p_0, \ldots, p_q \}
\subset \mathbf{SU}(N)  
\] 
be a set of ordered (knot) points in the Lie group $\mathbf{SU}(N)$, 
where $p_0=\psi_0$ and $p_j=\phi_j \psi_0^{-1}$, $\forall\, j\in\{1,\cdots,q\}$. 
Define the set of admissible maps by 
\begin{align*} 
\Phi_{\mathcal{P}}
=& \big\{ 
(U, V_{1},\ldots, V_{q}) \big | \, 
U:[x_0, x_q]\rightarrow \mathbf{SU}(N), V_{l} : [x_{l-1}, x_l] \to \mathbf{SU}(N), 
~\forall\, l \in \{1, \ldots, q\}, 
\\ 
&~~ 
U\in C^{2}([x_0, x_q]), 
U_l \in C^{4,\alpha_1}([x_{l-1}, x_l]), 
V_l \in C^{2,\alpha_2}([x_{l-1}, x_l]), 
U_{1}(x_{0})=p_{0},
V_{q}(x_q)=U_{q}(x_q), 
\\ 
&~~ 
\partial_x U_1(x_0)=\varphi^\prime_0, 
V_l(x_{l-1})=p_l, ~\forall\, l=1, \ldots, q, 
V_{l}(x_l)=U_{l}(x_l)=U_{l+1}(x_{l}), 
~\forall\, l=1, \ldots, q-1   
\big\}, 
\end{align*} 
where $x_l=l,\,\forall\, l\in\{0,1,\ldots,q\}$, 
$I_{l}=(x_{l-1}, x_{l})$, 
$\alpha_1, \alpha_2 \in (0,1)$, 
$U_l=U_{|_{\bar I_l}}$, 
and $\varphi^\prime_0=-\mathrm{i}H_0 U_0$. 
Note that $U \in C^k (\bar I)$ means  
$\partial_{x}^\ell U \in C^{0}(\bar I)$, 
$\forall\, \ell \in\{0,1,\ldots, k\}$.

To handle the penalty term in \eqref{eq:Action_J}, we replace it with 
\begin{align*}
\frac{1}{\sigma^2} \sum_{l=1}^{q} \mathcal{T}(V_l)
= \frac{1}{2\sigma^2} \sum_{l=1}^{q}
\int_{I_l} \lvert \partial_x V_l(x) \rvert^2 \, dx,
\end{align*}
where, for each $l$, $V_l$ is a curve connecting $p_l$ and $U(x_l)$. 
When $V_l$ is a critical point of the tension energy functional $\mathcal{T}$, it is a geodesic in the underlying Riemannian manifold, and in this case $\mathcal{T}$ satisfies
\begin{align*}
\mathcal{T}(V_l)
= \frac{1}{2|I_l|}\bigl(\operatorname{Length}(V_l)\bigr)^2 .
\end{align*}
Note that if $V_l$ is a geodesic in a Riemannian manifold $M$ with sufficiently small length, then 
$\text{dist}_{M}(V_l(x=0),V_l(x=1))=\text{Length}(V_l)$. 
Consequently, we study the energy functional $\mathcal{F}_{\sigma}
[(U, V_{1}, \ldots, V_{q} )]$ 
defined in \eqref{eq:F-U-sigma-energy}.

In this article, we replace the energy functional in \eqref{eq:Action_J_Q} by 
$\mathcal{F}_{\sigma}: 
\Phi_{\mathcal{P}}
\rightarrow [0,\infty)$, defined as 
\begin{align} 
\label{eq:F-U-sigma-energy} 
\mathcal{F}_{\sigma}
[(U, V_{1}, \ldots, V_{q} )] 
&= \mathcal{F}(U)+ \frac{1}{2\sigma^2} \sum\limits_{l=1}^{q} \int_{I_l} | V_{l,x}(x)|^2 \, dx, 
\end{align} 
and study the associated gradient flow, which was derived in \cite{LT26} in a more general higher-order setting.
The variational quantum splines introduced in \cite{BHM12} are then analyzed by resolving the initial–boundary value problem for the gradient flow of \eqref{eq:F-U-sigma-energy}. 
Namely, we consider the initial-boundary conditions of the gradient flow of 
$U: [0, T] \times [x_0, x_q]\rightarrow \mathbf{SU}(N)$, 
$V_l: [0, T] \times [x_{l-1}, x_{l}]\rightarrow \mathbf{SU}(N)$, and $\forall\, l\in\{1,\cdots,q\}$, 
set up as an intrinsic version: 
\begin{align}\label{U_flow} 
\partial_t U_{l}
=&\mathcal{L}^4_x(U_l) 
:=-D_x^{3} \partial_x U_{l} - R(D_x \partial_x U_{l}, \partial_x U_{l}) \partial_x U_{l}, ~~~~~~~~~~~~~~~~~~~~\text{ in }  [0,T] \times I_l, 
\\ 
\label{V_flow}
\partial_t V_{l}=&\sigma^2 D_x \partial_x V_{l}
~~~~~~~~~~~~~~~~~~~~~~~~~~~~~~~~~~~~~~~~~~~~~~~~~~~~~~~~~~~~~~~~~
\text{ in }  [0,T] \times I_l, 
\end{align}  
where $R(\cdot,\cdot)\cdot$ denotes the Riemannian curvature tensor of $\mathbf{SU}(N)$ and is induced by the bi-invariant metric of the Lie group $\mathbf{SU}(N)$, 
with the initial-boundary conditions:  
\begin{align}
\label{Br_initial-condition}
&(U,V_{1},\cdots,V_{q})(0,\cdot)
=(U_{0}(\cdot), V_{1,0}(\cdot), \cdots, V_{q,0}(\cdot) ) \in \Theta_{\mathcal{P}}, 
\\ 
\label{Br_clampend-pass-point} 
&U_{1}(t, x_0)=p_0,
~~~~~~~~~~~~~~~~~~~~~~~~~~~~~~~~~~~~~~~~~~~~~~~~~~~~~~~~~~~~~~~~~~~~~~~~~~~~~~~~~~~~~ 
t \in [0, T], 
\\ 
\label{Br_clampend-first-second-derivative}
&
\partial_x U_{1}(t, x_0)=\varphi^\prime_0 
~~ \text{(a constant)}, 
~~~~~~~~~~~~~~~~~~~~~~~~~~~~~~~~~~~~~~~~~~~~~~~~~~~~~~~~~~~~~~~~ 
t \in [0, T], 
\\
\label{Br_boundary-condition-1}
& V_l(t,x_{l-1})=p_l, ~~~~~~~~~~~~~~~~~~~~~~~~~~~~~~~~~~~~~~~~~~~~~~~~~~~~~~~~~~~~~~~~~ l \in \{1, \ldots, q\}, t \in [0, T], 
\\
\label{Br_boundary-condition-2}
&U_{l}(t, x_l)=V_{l}(t, x_l)=U_{l+1}(t, x_l), ~~~~~~~~~~~~~~~~~~~~~~~~~~~~~~~~~~~~~~ 
l \in \{1, \ldots, q-1\}, t \in [0, T], 
\\
\label{Br_boundary-condition-3}
&\Delta_{x_l}[D_x^{\mu-1} U_x](t)=0, 
~~~~~~~~~~~~~~~~~~~~~~~~~~~~~~~~~~~~~~~ 
\mu \in \{1, 2\}, l \in \{1, \ldots, q-1\}, t \in [0, T], 
\\
\label{Br_boundary-condition-4}
&\Delta_{x_l}[D_x^2 \partial_x U](t)+\frac{1}{\sigma^2} \partial_x V_{l}(t, x_l)=0, 
~~~~~~~~~~~~~~~~~~~~~~~~~~~~~~~~~~ 
l \in \{1, \ldots, q-1\}, t \in [0, T], 
\\ 
\label{Br_boundary-condition-5}
&V_{q}(t, x_q)=U_{q}(t, x_q), ~~~~~~~~~~~~~~~~~~~~~~~~~~~~~~~~~~~~~~~~~~~~~~~~~~~~~~~~~~~~~~~~~~~~~~~~~~~~~ 
t \in [0, T],
\\
\label{Br_boundary-condition-6}
&  
D_x \partial_x U_q(t, x_q)=0, 
~~~\text{ or }~~~ 
\partial_x U_q(t, x_q)=\varphi^\prime_q 
~~ \text{(a constant)}, 
~~~~~~~~~~~~~~~~~~~~~~~~~~~~~  
t \in [0, T], 
\\
\label{Br_boundary-condition-7} 
&-D_x^2\partial_x  U_q(t, x_q)+\frac{1}{\sigma^2} \partial_x V_{q}(t, x_q)=0,  
~~~~~~~~~~~~~~~~~~~~~~~~~~~~~~~~~~~~~~~~~~~~~~~~~~~~~~~~
t \in [0, T].
\end{align} 
Here, $U_0$ denotes the initial data of $U$ and 
\begin{align*}
[\Delta_{l} D_x^{\mu-1} \partial_x U](t) 
:=D_x^{\mu-1} \partial_xU_{l+1}(t, x_{l})
-D_x^{\mu-1} \partial_x U_{l}(t, x_{l}), 
\forall\, \mu\in \mathbb{N}. 
\end{align*}

It is worth mentioning here that, when $M$ is a compact and connected Lie group with 
a bi-invariant metric, we have the following identity for the Riemannian curvature tensors: 
\begin{equation} 
\label{eq:Riem_Curv_Tensor_of_Lie_Grp}
R_M(X,Y)Z=-\frac{1}{4}[[X,Y],Z], 
\end{equation} 
where $D$ is the Levi–Civita connection induced from the bi-invariant Riemannian metric on $\mathbf{SU}(N)$, and $X, Y, Z \in TM$ are left-invariant vector fields on $\mathbf{SU}(N)$ 
(e.g., see \cite[Lemma 2]{BCC21}).

\bigskip

Below, we verify the following identities for the covariant derivatives of $U$:
\begin{empheq}[left=\empheqlbrace\quad]{align}
\label{eq:d^2_x(U_x)=new} 
\begin{split}
D_x U_x
&=-\mathrm{i} H_{x} U,
\\
D_x^2 U_x&= \left(- \mathrm{i} H_{xx} - \frac{1}{2}[H_x, H]\right)U, 
\\ 
D_x^3 U_x 
&=\left(-\mathrm{i} H_{xxx}+[H, H_{xx}]+\frac{1}{4}[[H_x,H],\mathrm{i} H] \right)U.
\end{split}
\end{empheq} 
It is straightforward to derive from the definition of $\text{skew}_0$ that 
\begin{align}
\label{eq:skew_0(AB)=[A,B]/2}
\text{skew}_0(AB)=\frac{1}{2}[A,B], 
~~~ \text{ if both } A, B \text{ are Hermitian, or skew-Hermitian}. 
\end{align} 
By applying \eqref{H-Hermitian},  \eqref{eq:d_x(U_x)=}, 
\eqref{eq:skew_0(AB)=[A,B]/2}, 
\eqref{eq:skew(H)}, and 
using \eqref{D_xY=projection} with $Y=D_xU_x$, 
we have 
\begin{align*} 
D_x^2 U_x
&=D_x(D_x U_x)
=P(\partial_x (-\mathrm{i}H_x U)) 
=P(-\mathrm{i} H_{xx} U- \mathrm{i}H_x (-\mathrm{i}H U)) 
\\ 
&= \text{skew}_0\left(-\mathrm{i} H_{xx}-H_x H\right) U 
= \left(-\mathrm{i} H_{xx}-\frac{1}{2}[H_x, H]\right) U.
\end{align*} 
Note that $[H_x,H]^\ast=-[H_x,H]$ 
(i.e., skew-Hermitian). 
Together with applying \eqref{H-Hermitian}, \eqref{eq:d_x(U_x)=}, \eqref{eq:skew_0(AB)=[A,B]/2}, \eqref{eq:skew(H)}, 
and 
using \eqref{D_xY=projection} with $Y=D^2_xU_x$ again, 
we obtain
\begin{align*} 
D_x^3 U_x
&=D_x(D^2_x U_x)
=P(\partial_x ((- \mathrm{i} H_{xx} - \frac{1}{2}[H_x, H])U)) 
\\ 
&=P((- \mathrm{i}H_{xxx} - \frac{1}{2}[H_x, H]_x)U)
+P((-\mathrm{i} H_{xx} - \frac{1}{2}[H_x, H])(-\mathrm{i} HU)) 
\\ 
&= \text{skew}_0\left(- \mathrm{i} H_{xxx} - \frac{1}{2}[H_x, H]_x -H_{xx}H + \frac{1}{2}[H_x, H](\mathrm{i} H)\right) U 
\\ 
&= \left(- \mathrm{i} H_{xxx} - \frac{1}{2}[H_x, H]_x -\frac{1}{2}[H_{xx},H] + \frac{1}{4}[[H_x, H], \mathrm{i} H]\right)  U.
\end{align*}
Finally, observe that 
\begin{align*}
- \frac{1}{2}[H_x, H]_x -\frac{1}{2}[H_{xx},H]
=&-\frac{1}{2}\Big((H_xH-HH_x)_x+(H_{xx}H-HH_{xx})\Big) 
=[H, H_{xx}],
\end{align*}
which yields the claimed expression for $D_x^3 U_x$.
The equivalence between the Euler–Lagrange equations for Riemannian cubics in $\mathbf{SU}(N)$ and for quantum splines, as described in  \cite{BHM12}, 
then follows by applying
\eqref{eq:Riem_Curv_Tensor_of_Lie_Grp} and \eqref{eq:d^2_x(U_x)=new}, namely, we obtain 
\begin{align} 
\nonumber 
D_x^3 U_x +R(D_xU_x, U_x)U_x 
=&\left(-\mathrm{i} H_{xxx}+[H, H_{xx}]+\frac{1}{4}[[H_x,H],\mathrm{i} H] \right)U-\frac{1}{4}[[-\mathrm{i} H_x,-\mathrm{i} H],-\mathrm{i} H]U 
\\
=&\left(-\mathrm{i} H_{xxx}+[H, H_{xx}] \right)U 
\text{.} 
\label{eq:Riem_Curv_Tensor} 
\end{align}

In the following definition 
(Definition \ref{def:compatibility_cond_order_zero}),  
we formulate the compatibility conditions of order $0$. These conditions ensure that the prescribed initial data are consistent with all boundary constraints imposed on $(U, V_1,\ldots, V_q)$. 
The compatibility conditions of order $0$ are required to guarantee that the solutions remain regular up to $t=0$; in particular, the functions  $\partial_t U_l$, $\{\partial_x^j U_l: j=0,\ldots,4\}$, $\partial_t V_l$, and 
$\{\partial_x^j V_l: j=0,\ldots,2\}$ 
extend continuously to the parabolic boundary.
For further background and related discussions of the compatibility conditions of order $0$, we refer the reader to \cite[Remark 3.4]{LT26}.

\begin{definition}[The compatibility conditions of order $0$]
\label{def:compatibility_cond_order_zero}
An initial datum $(U_0, V_{1,0},\ldots, V_{q,0})$ is said to fulfill the compatibility conditions of order $0$ for the gradient flow of $(U, V_{1},\ldots, V_{q})$ in \eqref{U_flow} and \eqref{V_flow} if the followings hold: 
\begin{equation}
\label{eq:cc-order-0-fitting}
\begin{cases} 
\mathcal{L}^4_x(U_{1, 0})(x_0)=0, 
\\
D_x \partial_x V_{l, 0}(x_l)=0, ~~~~~~~~~~~~~~~~~~~~~~~~~~~~~~~~~~~~~~~~~~~~~~~~~~~~~~~~~~~~~~~~~~~~~ 
l \in \{1, \ldots, q\}, 
\\
\mathcal{L}^4_x(U_{l, 0})(x_l)=\sigma^2  D_x \partial_x V_{l, 0}(x_l)=\mathcal{L}^4_x(U_{l+1, 0})(x_l), ~~~~~~~~~~~~~~~~~~~~~~~~~ 
l \in \{1, \ldots, q-1\}, 
\\ 
\mathcal{L}^4_x(U_{q, 0})(x_q)=\sigma^2 D_x \partial_x V_{q, 0}(x_q), 
\\
U_{1,0}(x_0)=p_0,  
\\
V_{l,0}(x_{l-1})=p_l, ~~~~~~~~~~~~~~~~~~~~~~~~~~~~~~~~~~~~~~~~~~~~~~~~~~~~~~~~~~~~~~~~~~~~~~~~ 
 l \in \{1, \ldots, q\}, 
\\
U_{l,0}(x_l)=V_{l,0}(x_l)=U_{l+1,0}(x_l), ~~~~~~~~~~~~~~~~~~~~~~~~~~~~~~~~~~~~~~~~~~~~~~~ 
l \in \{1, \ldots, q-1\},  
\\
V_{q,0}(x_q)=U_{q,0}(x_q),
\\ 
\partial_x U_{1, 0}(x_0)=\varphi^\prime_0, 
~~\text{(a constant)},  
\\
\Delta_{x_l}[D_x^{\mu-1} \partial_x U_{0}]=0, 
~~~~~~~~~~~~~~~~~~~~~~~~~~~~~~~~~~~~~~~~~~~~~~ 
\mu \in \{1, 2\}, l \in \{1, \ldots, q-1\},  
\\
\Delta_{x_l}[D_x^2 \partial_x U]+\frac{1}{\sigma^2} \partial_x V_{l,0}(x_l)=0, 
~~~~~~~~~~~~~~~~~~~~~~~~~~~~~~~~~~~~~~~~~~~~~~ 
l \in \{1, \ldots, q-1\},  
\\ 
D_x\partial_x U_{q,0}(x_q)=0 
~~~\text{ or }~~~ 
\partial_x U_{q,0}(x_q)=\varphi^\prime_q 
~~\text{(a constant)},  
\\
-D_x^2\partial_x  U_{q,0}(x_q)+\frac{1}{\sigma^2} \partial_x V_{q,0}(x_q)=0.  
\end{cases}
\end{equation} 
\end{definition}

By following the same argument in \cite{LT26}, the energy identity is a direct consequence as the solution $(U,V_1,\cdots,V_q)$ to the gradient flow remains in the proper parabolic H\"{o}lder space or parabolic Sobolev spaces. 
More precisely, by using the gradient flow  \eqref{U_flow}$\sim$\eqref{V_flow} 
and applying the boundary conditions 
\eqref{Br_clampend-pass-point}$\sim$\eqref{Br_boundary-condition-7} 
through integration by parts,  
we obtain
\begin{align*}
\frac{d}{dt}\mathcal{F}_\sigma\big[(U, V_{1}, \ldots, V_{q})\big]
=
-\sum_{l=1}^{q} \int_{I_l} |U_{l,t}|^2 \, dx
-\frac{1}{\sigma^2}\sum_{l=1}^{q} \int_{I_l} |V_{l,t}|^2 \, dx
\;\le\; 0 .
\end{align*}
As a direct consequence of the above energy identity, we deduce the following
uniform-in-time estimates:
\begin{equation}
\label{eq:(U+V)_{xx}_bdd}
\begin{cases}
\sum\limits_{l=1}^{q}  \|D_x U_{l,x}\|_{L^2(I_l)} (t) \leq \mathcal{F}_\sigma(0), 
\\ 
\sum\limits_{l=1}^{q}  \|V_{l,x}\|_{L^2(I_l)} (t) \leq \sigma^2 \mathcal{F}_\sigma(0),
\end{cases}
\end{equation}
where 
\[
\mathcal{F}_\sigma(0):=\mathcal{F}_\sigma[(U(0,\cdot), V_{1}(0,\cdot), \ldots, V_{q}(0,\cdot) )].
\]
Since $U(t,x)\in \mathbf{SU}(N)$ for all $(t,x)$, we have $U^*(t,x)U(t,x)=\operatorname{Id}_{N \times N}$. 
It follows immediately that
\[
|U(t,x)|^2
= \operatorname{tr}\!\big(U^*(t,x)U(t,x)\big)
= \operatorname{tr}(\operatorname{Id}_{N \times N})
= N.
\]
Hence $U$ is uniformly bounded with respect to the norm $| \cdot |$ in \eqref{def:norm-matrix}, and in fact 
\begin{equation}
\label{eq:U_{sup}-bdd}
|U(t,x)| = \sqrt{N}
\quad \text{for all } (t,x)\in [0, T]\times [x_0, x_q].
\end{equation} 
By applying \cite[Lemma 5.11]{LT26} (ii) with the estimates in \eqref{eq:U_{sup}-bdd} and \eqref{eq:(U+V)_{xx}_bdd}, 
we obtain 
\begin{align*}
\sum_{l=1}^{q}\|U_{l,x}\|_{L^2(I_l)}(t)\le C(\mathcal{F}_\sigma(0), N). 
\end{align*}

\begin{theorem}
\label{thm:Main_Thm_1} 
Let $\sigma \in (0,\infty)$ and $\alpha_1, \alpha_2 \in(0,1)$ with $\alpha_1>2 \alpha_2$.
Suppose the initial datum 
$(U_0, V_{1,0}, \ldots, V_{q,0})\in \Theta_{\mathcal{P}}$ 
satisfies the compatibility conditions of order $0$ for the gradient flow of $(U, V_{1}, \ldots, V_{q})$ in \eqref{U_flow} and \eqref{V_flow}. 
Then, there exists a global solution $(U, V_1, \ldots, V_{q})$ to 
\eqref{U_flow}$\sim$\eqref{Br_boundary-condition-7} 
with the following regularity properties: 
\begin{equation*}
\begin{cases}
U_{l} 
\in 
C^{\frac{4+\alpha_1}{4},4+\alpha_1}
\left([0,\infty)\times [x_{l-1}, x_{l}]\right) \bigcap C^{\infty}
\left((0,\infty)\times [x_{l-1}, x_{l}]\right), 
~~~~~~~~~
\forall\, l\in\{1,\ldots,q\}, 
\\
V_{l} 
\in 
C^{\frac{2+\alpha_2}{2},2+\alpha_2}
\left([0,\infty)\times [x_{l-1}, x_{l}]\right) \bigcap C^{\infty}
\left((0,\infty)\times [x_{l-1}, x_{l}]\right), 
~~~~~~~~~
\forall\, l\in\{1,\ldots,q\}, 
\\
U(t,\cdot) \in C^{2}([x_{0}, x_{q}]), ~~~~~~~~~~~~~~~~~~~~~~~~~~~~~~~~~~~~~~~~~~~~~~~~~~~~~~~~~~~~~~~~~~~ 
\forall\, t\in[0,\infty).
\end{cases}
\end{equation*}  
Furthermore, there exists $(U_{\infty}, V_{1, \infty}, \ldots, V_{q, \infty}) \in \Theta_{\mathcal{P}}$ with the regularity $U_{\infty} \in C^{2}([x_0, x_q])$, which is obtained from some convergent subsequences $U_{\infty}(\cdot)=\underset{t_j\rightarrow\infty}\lim U(t_j,\cdot)$ and $V_{l, \infty}(\cdot)=\underset{t_j\rightarrow\infty}\lim V_{l}(t_j,\cdot)$,  such that $U_{l, \infty}$ is a Riemannian cubic and $V_{l, \infty}$ behaves as a geodesic map, $\forall\, l$. 
\end{theorem}

\section{Solutions to the Gradient Flow } \label{sub:solution-gradient-flow}

For clarity and completeness, we present the following outline of the proof, which establishes the existence of solutions for the governing nonlinear parabolic system.

\begin{description}
\item[1.] Preparatory Steps (\S\ref{sec:Pre}). 

These steps set up the analytical framework underlying the entire proof.
\begin{itemize}
\item Use of Standard PDE Tools: 
The chosen extrinsic formulation permits the direct application of the classical theory for parabolic systems, most notably Solonnikov’s theory, which is fundamental for establishing existence results.

\item Systematic Treatment of Boundary Conditions: 
A coherent organization of the boundary conditions within the parabolic framework is achieved by reparametrizing certain curve segments with reversed orientations and introducing suitable notation.
\end{itemize}

\item[2.] Short-Time Existence or Local Solutions (\S\ref{sec:STE}). 

Local existence is proved through a two-stage argument.
\begin{itemize}
\item Linearization and Application of Solonnikov’s Theory: 
As a first step, we linearize the nonlinear parabolic system to which Solonnikov's theory will be applied. 
Because the components $g_l$ and $h_l$ possess \emph{different orders of parabolicity}, the linearized problem splits into two linear subsystems coupled through their boundary conditions.
A key step is the derivation of the necessary \emph{a priori} estimates for the linearized system. The main technical difficulty lies in verifying the \emph{complementary conditions}, a subtle but essential requirement for well-posedness.

\item Banach Fixed Point Argument: 
Using the \emph{a priori} estimates for the linearized problem, we apply the Banach Fixed Point Theorem to construct short-time solutions of the full nonlinear system.
To obtain the desired regularity up to the initial time, especially at the boundary, it is crucial to impose compatibility conditions on the initial data. These conditions enable the bootstrapping procedure and ensure that Solonnikov’s theory yields the higher regularity required for the local solutions.
\end{itemize}

\item[3.] Long-Time Existence or Global Solutions (\S\ref{sec:LTE}). 

Global existence follows from uniform estimates.
\begin{itemize}
\item Uniform Higher-Order Bounds.
We establish uniform bounds on higher-order derivatives of $U_{l}$ and $V_{l}$, valid for all $t$ and each fixed $l$, by combining Gagliardo–Nirenberg interpolation inequalities with Gr\"{o}nwall’s inequality.
These estimates allow us to implement a contradiction argument, which leads to long-time existence.

\item Asymptotic Analysis.
With the uniform bounds in place, we study the asymptotic behaviour of solutions and prove the existence of limits of $U(t_j,\cdot)$ along suitable convergent subsequences.
\end{itemize}

\end{description}

\subsection{The Preparatory Steps} 
\label{sec:Pre}

We rewrite the system \eqref{U_flow}$\sim$\eqref{Br_boundary-condition-7} in the following extrinsic formulation: 
\begin{equation}
\begin{cases}
\partial_t U_{l}=- \partial_x^4 U_{l} + G(\partial_x^3 U_{l}, \partial_x^2 U_{l}, \partial_x U_{l}, U_l), ~~~~~~~~~~~~~~~~~~~~~~~~~~~~~~~~~~~~~~~~
\text{ in }  [0,T] \times I_l, 
\\
\partial_t V_{l}=\sigma^2 \partial_x^2 V_{l}+ \sigma^2 b_1(\partial_x V_l, V_l), ~~~~~~~~~~~~~~~~~~~~~~~~~~~~~~~~~~~~~~~~~~~~~~~~~~~~
\text{ in }  [0,T] \times I_l, 
\\
(U,V_{1},\cdots, V_{q})(0,\cdot)=(U_{0}(\cdot), V_{1,0}(\cdot),\cdots, V_{q,0}(\cdot) ) \in \Theta_{\mathcal{P}}, 
\\
U_0(t, x_0)=p_0, 
\\
\partial_x U_1(t, x_0)=\varphi^{\prime}_0 
~~ \text{(a constant)},  
\\
V_l(t, x_{l-1})=p_l, ~~~~~~~~~~~~~~~~~~~~~~~~~~~~~~~~~~~~~~~~~~~~~~~~~~~~~~~~~~~~~~~~~~~~~~~~
l \in \{1, \ldots, q\}, 
\\
U_{l}(t, x_l)=V_{l}(t, x_l)=U_{l+1}(t, x_l), ~~~~~~~~~~~~~~~~~~~~~~~~~~~~~~~~~~~~~~~~~~~~~ 
l \in \{1, \ldots, q-1\}, 
\\
\partial_x^{\mu} U_l(t, x_l)-\partial_x^{\mu} U_{l+1}(t, x_l)=0,
~~~~~~~~~~~~~~~~~~~~~~~~~~~~~~~~~ 
\mu \in \{1, 2\}, l \in \{1, \ldots, q-1\}, 
\\
-\partial_x^{3} U_l(t, x_l)+\partial_x^{3} U_{l+1}(t, x_l)+\frac{1}{\sigma^2} \partial_x V_{l}(t, x_l) 
\\ ~~~~~~~~~~ 
=-b_2(\partial_x^2 U_{l+1}, \partial_x U_{l+1}, U_{l+1} )(t, x_l) 
+ b_2(\partial_x^2 U_l, \partial_x U_{l}, U_{l})(t, x_l),
~ l \in \{1, \ldots, q-1\}, 
\\ 
V_{q}(t, x_q)=U_{q}(t, x_q), 
\\ 
\partial_x^2 U_q(t, x_q)= b_1(\partial_xU_q,  U_q)(t, x_q) 
~~ \text{ or } ~~ 
\partial_x U_q(t, x_q)=\varphi^{\prime}_q 
~~ \text{(a constant)}, 
\\ 
\partial_x^3  U_q(t, x_q)=\frac{1}{\sigma^2} \partial_x V_{q}(t, x_q) -b_2(\partial_x^2 U_q, \partial_x U_{q},  U_{q})(t, x_q).
\end{cases} 
\end{equation} 
where 
\begin{align}  
\nonumber 
G(\partial_x^3 W, \partial_x^2 W, \partial_x W, W) 
:=&2\, \partial_x^3 W \cdot  W^\ast \cdot \partial_x W
+2\, \partial_x W \cdot  W^\ast \cdot \partial_x^3 W 
-4\,\partial_x^2 W  \cdot (W^\ast \cdot \partial_x W)^2
\\\nonumber 
&-4\,W^\ast \cdot \partial_x W \cdot \partial_x^2 W  \cdot W^\ast \cdot \partial_x W
-4\, (\partial_x W \cdot W^\ast)^2 \cdot \partial_x^2 W\\
&
+3\, \partial_x^2 W \cdot W^\ast \cdot \partial_x^2 W
+6\,(\partial_x W  \cdot W^\ast)^3 \cdot \partial_x W,
\label{eq:def_G} 
\\ 
\label{def:b_1}
b_1(\partial_x W, W)
:=& -\partial_x W \cdot W^{\ast} \cdot \partial_x W, 
\\
\label{def:b_2}
b_2(\partial_x^2 W, \partial_x W, W)
:=& -\frac{3}{2}\partial_x^2 W \cdot W^{\ast} \cdot \partial_x W+2 (\partial_x W \cdot W^{\ast})^2 \cdot \partial_x W -\frac{3}{2}\partial_x W \cdot W^{\ast} \cdot \partial_x^2 W, 
\end{align} 
for all $W\in \mathbf{SU}(N)$.  
Note that we obtain \eqref{eq:def_G} using \eqref{U_flow}, \eqref{eq:Riem_Curv_Tensor}, \eqref{eq:Schoedinger_Eq}, and \eqref{eq:partial_x(U*)}; \eqref{def:b_1} follows from \eqref{eq:Schoedinger_Eq}, \eqref{eq:partial_x(U*)}, and \eqref{eq:d_x(U_x)=}; and \eqref{def:b_2} from \eqref{eq:Schoedinger_Eq}, \eqref{eq:partial_x(U*)}, and the second identity in \eqref{eq:d^2_x(U_x)=new}.

Using the assumption $x_l=l$ for all $l\in\mathbb{Z}$, 
we introduce the parametrization: 
\begin{align} 
\label{eq:parametrization-gamma}
g_{l}(t,y) = U_{l}\left(t, (-1)^{l-1}y + x_{2\left[\frac{l}{2}\right]}\right),
\quad g_{l,0}(y) = U_{l,0}\left((-1)^{l-1}y + x_{2\left[\frac{l}{2}\right]}\right), 
\end{align}  
\begin{align}
\label{eq:parametrization-chi}
h_{l}(t,y)=V_{l}(t, (-1)^{l-1}y+ x_{2\left[\frac{l}{2}\right]}), 
\quad 
h_{l,0}(y)=V_{l,0}((-1)^{l-1}y+x_{2\left[\frac{l}{2}\right]}).   
\end{align}  
For even $l$, the orientations of $g_{l}$ and $h_{l}$  are obtained by reversing those of $U_{l}$ and $V_{l}$, respectively. 
This transformation allows us to place the problem into a form amenable to Solonnikov’s theory and the Banach Fixed Point Theorem, which are used to establish short-time existence. 
Applying \eqref{eq:parametrization-gamma}$\sim$\eqref{eq:parametrization-chi}, the evolution equations for $(U, V_{1}, \ldots, V_{q})$ in
\eqref{U_flow}$\sim$\eqref{Br_boundary-condition-7}
can be rewritten as a nonlinear parabolic system for $(g,h)$: 
\begin{equation}
\label{eq:nonlinear-g-h_l}
\begin{cases}
\partial_t g_{l}+\partial_y^{4} g_{l}
=G(\partial_y^{3} g_{l}, \partial_y^2 g_{l}, \partial_y g_{l}, g_{l}), 
~~~~~~~~~~~~~~~~~~~~~~~~~~~~~~~~~~~~~~~~~~~~~~ 
l \in\{1, \ldots, q\}, 
\\ 
\partial_t h_{l}= \sigma^2 \partial_y^{2} h_{l} 
+ \sigma^2  b_1(\partial_y h_l, h_l) ~~~~~~~~~~~~~~~~~~~~~~~~~~~~~~~~~~~~~~~~~~~~~~~~~~~~~ 
l \in\{1, \ldots, q\}, 
\\ 
g_{l}(0,y)= g_{l,0}(y), 
~~~~~~~~~~~~~~~~~~~~~~~~~~~~~~~~~~~~~~~~~~~~~~~~~~~~~~~~~~~~~~~~~~~~~ 
l\in \{1, \ldots, q\}, 
\\
g_{1}(t,0)= p_{0},  
\\ 
\partial_{y}g_{1}(t,0)=\varphi^{\prime}_0 
~~ \text{(a constant)},  
\\ 
g_{l}(t, y^{\ast}_{l})
=h_{l}(t, y^{\ast}_{l})
=g_{l+1}(t, y^{\ast}_{l}), 
~~~~~~~~~~~~~~~~~~~~~~~~~~~~~~
l \in \{1, \ldots, q-1\}, ~
 y^{\ast}_l= 2\{\frac{l}{2}\},  
\\ 
\partial_{y}^{\mu-1}g_{l}(t,  y^{\ast}_{l})
+ (-1)^{\mu-1} \partial_{y}^{\mu} g_{l+1}(t, y^{\ast}_{l})=0, 
~~~~
\mu \in \{1, 2\}, 
~ l \in \{1, \ldots, q-1\}, 
~ y^{\ast}_{l}= 2\{\frac{l}{2}\},  
\\
\partial_{y}^{3} g_{l}(t, y^{\ast}_{l})
+\partial_{y}^{3} g_{l+1}(t, y^{\ast}_{l})+\frac{1}{\sigma^2}\partial_{y} h_{l}(t, y^{\ast}_l)=b_2(\partial_{y}^{2} g_l, \partial_{y} g_l,  g_l)(t,  y^{\ast}_{l}) 
\\
~~~~~~~~~~~~~~~~~~~~~~~~~~~~~~ 
-b_2(\partial_{y}^{2}g_{l+1}, \partial_{y} g_{l+1},  g_{l+1})(t,  y^{\ast}_{l}), 
~~~ 
l \in \{1, \ldots, q-1\}, ~
y^{\ast}_{l}= 2\{\frac{l}{2}\}, 
\\ 
g_{q}(t, 2 \big\{ \frac{q}{2}\big\})
= h_{q}(t, 2 \big\{ \frac{q}{2}\big\}), 
\\ 
\partial_{y}^2 g_{q}(t, 2 \big\{ \frac{q}{2}\big\})
= b_1( \partial_y g_q, g_q)(t, 2 \big\{ \frac{q}{2}\big\}) 
~~~ \text{ or } ~~~ 
\partial_{y}g_{q}(t,2 \big\{ \frac{q}{2}\big\})=(-1)^{q-1}\varphi^{\prime}_q 
~~~ \text{(a constant)},  
\\
\partial_{y}^{3} g_{q}(t, 2 \big\{ \frac{q}{2}\big\})
+\frac{1}{\sigma^2}\partial_{y} h_{q}(t, 2 \big\{ \frac{q}{2}\big\})=b_2(\partial_y^2 g_q, \partial_y g_q, g_q)(t, 2 \big\{ \frac{q}{2}\big\}), 
\end{cases}
\end{equation} 
for all $t \in [0, T]$, where $b_1(\partial_y g_l, g_l)$ and $b_2(\partial_y^2 g_l, \partial_y g_l, g_l)$ are defined in \eqref{def:b_1} and \eqref{def:b_2}, respectively.

\subsection{The Short-Time Existence }
\label{sec:STE}

We introduce the spaces of functions used in the application of Solonnikov’s theory and 
the Banach fixed-point argument.
For $T>0$, define 
\[
Z^{T}
:=
\bigl(C^{\frac{4+\alpha_1}{4},\,4+\alpha_1}([0,T]\times[0,1])\bigr)^{q}
\times
\bigl(C^{\frac{2+\alpha_2}{2},\,2+\alpha_2}([0,T]\times[0,1])\bigr)^{q},
\]
endowed with the norm
\[
\|(g,h)\|_{Z^{T}}
=
\sum_{l=1}^{q}\|g_l\|_{C^{\frac{4+\alpha_1}{4},\,4+\alpha_1}}
+
\sum_{l=1}^{q}\|h_l\|_{C^{\frac{2+\alpha_2}{2},\,2+\alpha_2}}.
\]
Let $(g_0,h_0)$ denote the initial data associated with
$(U_0,V_{1,0},\ldots,V_{q,0})\in\Theta_{\mathcal P}$
via the parametrizations \eqref{eq:parametrization-gamma}
and \eqref{eq:parametrization-chi}.
We define 
\[
X^{T}_{(g_0,h_0)}
:=
\bigl\{(g,h)\in Z^{T}:
g_l(0,\cdot)=g_{l,0},\;
h_l(0,\cdot)=h_{l,0},
\; l=1,\ldots,q
\bigr\}, 
\]
which is a closed convex subset of $Z^{T}$.
For $K>0$, set 
\[
B_K
:=
\bigl\{(g,h)\in X^{T}_{(g_0,h_0)}: 
\|(g,h)\|_{Z^{T}}\le K
\bigr\}.
\]

Due to the different orders of parabolicity of $g_l$ and $h_l$ in \eqref{eq:nonlinear-g-h_l}, we decompose the linearized problem into two subsystems, to each of which Solonnikov's theory is then applied separately, as described below: 
\begin{equation}
\label{eq:higher-order-linear-g_l}
\begin{cases}
\partial_t g_{l}+\partial_y^{4} g_{l}
=G(\partial_y^{3} \bar{g}_{l}, \partial_y^2 \bar{g}_{l}, \partial_y \bar{g}_{l}, \bar{g}_{l}), 
~~~~~~~~~~~~~~~~~~~~~~~~~~~~~~~~~~~~~~~~~~~~~ 
l \in\{1, \ldots, q\}, 
\\ 
g_{l}(0,y)= g_{l,0}(y), 
~~~~~~~~~~~~~~~~~~~~~~~~~~~~~~~~~~~~~~~~~~~~~~~~~~~~~~~~~~~~~~~~~~~~ 
l\in \{1, \ldots, q\}, 
\\ 
g_{1}(t,0)= p_{0},  
\\ 
\partial_{y}g_{1}(t,0)=\varphi^{\prime}_0 
~~ \text{(a constant)}, 
\\ 
g_{l}(t, y^{\ast}_{l})
=g_{l+1}(t, y^{\ast}_{l}), 
~~~~~~~~~~~~~~~~~~~~~~~~~~~~~~~~~~~~~~~~~~~
l \in \{1, \ldots, q-1\}, ~
 y^{\ast}_l= 2\{\frac{l}{2}\},   
\\ 
\partial_{y}^{\mu-1}g_{l}(t,  y^{\ast}_{l})
+ (-1)^{\mu-1} \partial_{y}^{\mu} g_{l+1}(t, y^{\ast}_{l})=0, 
~~~
\mu \in \{1, 2\}, 
~ l \in \{1, \ldots, q-1\}, 
~ y^{\ast}_{l}= 2\{\frac{l}{2}\},   
\\ 
\partial_{y}^{3} g_{l}(t, y^{\ast}_{l})
+\partial_{y}^{3} g_{l+1}(t, y^{\ast}_{l})=\frac{1}{\sigma^2}\partial_{y} \bar{h}_{l}(t, y^{\ast}_l)+b_2(\partial_{y}^2 \bar{g}_l, \partial_{y}\bar{g}_l, \bar{g}_l)(t, y^{\ast}_{l}) 
\\
~~~~~~~~~~~~~~~~~~~~~~~~~~~~~~~ 
-b_2(\partial_{y}^2 \bar{g}_{l+1}, \partial_{y}\bar{g}_{l+1}, \bar{g}_{l+1})(t,y^{\ast}_{l}), ~ 
l \in \{1, \ldots, q-1\}, 
~ y^{\ast}_{l}= 2\{\frac{l}{2}\},  
\\ 
\partial_{y}^2g_{q}(t,2 \big\{ \frac{q}{2}\big\})= b_1 (\partial_y \bar{g}_q, \bar{g}_q)(t,2 \big\{ \frac{q}{2}\big\}) 
~~~ \text{ or } ~~~ 
\partial_{y}g_{q}(t,2 \big\{ \frac{q}{2}\big\})=(-1)^{q-1}\varphi^{\prime}_q ~~ \text{(a constant)}, 
\\
\partial_{y}^{3} g_{q}(t, 2 \big\{ \frac{q}{2}\big\})
=- \frac{1}{\sigma^2}\partial_{y} \bar{h}_{q}(t, 2 \big\{ \frac{q}{2}\big\})+b_2(\partial_{y}^2 \bar{g}_q, \partial_{y}\bar{g}_q, \bar{g}_q)(t, 2 \big\{ \frac{q}{2}\big\}), 
\end{cases}
\end{equation} 
and 
\begin{equation}
\label{eq:second-order-linear-h_l}
\begin{cases}
\partial_t h_{l}= \sigma^2 \partial_y^{2} h_{l} 
+ \sigma^2 b_1(\partial_y \bar{h}_l, \bar{h}_l), 
\\ 
h_{l}(0,y)= h_{l,0}(y), 
\\
h_{l}(t, y^{\ast}_l)= p_{l}, 
~~~~~~~~~~~~~~~~~~~~~~~~~~~~~~~~~~~~~~~~~~~~~~~~~~~~~~~~~~~~~~~~~~ 
y^{\ast}_{l}= 2\{\frac{l}{2}\},   
\\
 h_{l}(t, y^{\ast}_l)=g_{l}(t, y^{\ast}_{l}),
~~~~~~~~~~~~~~~~~~~~~~~~~~~~~~~~~~~~~~~~~~~~~~~~~~~~~~~~~~~ 
y^{\ast}_{l}= 2\{\frac{l}{2}\}, 
\end{cases}
\end{equation} 
where 
$g_l$ in \eqref{eq:second-order-linear-h_l} is a solution to \eqref{eq:higher-order-linear-g_l}, 
$\forall\, l \in \{1, \ldots, q\}$, $t\in[0,T]$, $y\in[0,1]$, $\{\frac{b}{a}\}:=\frac{b}{a}-[\frac{b}{a}]\in[0,1)$, $(\bar{g}, \bar{h}) \in X^{T}_{(g_0,h_0)}$.


\begin{theorem}[Solutions to the coupled linear parabolic systems, \eqref{eq:higher-order-linear-g_l}$\sim$\eqref{eq:second-order-linear-h_l}]
\label{thm:STE_for_linear-fitting} 
Let $\sigma \in (0,\infty)$ and $\alpha_j\in(0,1)$, $j \in \{1, 2\}$ with $\alpha_1>2\alpha_2$. 
Suppose that $(g_0,h_0)$ is the initial datum for the nonlinear parabolic system \eqref{eq:nonlinear-g-h_l}, where  
$g_0=(g_{1,0},\cdots,g_{q,0})$,   
$h_0=(h_{1,0}, \ldots, h_{q,0})$, 
and each $g_{l,0}$ and $h_{l,0}: [0,1] \to \mathbb{R}^n$ satisfy the regularity conditions $g_{l,0}\in C^{4+\alpha_1}([0,1])$ and $h_{l,0} \in C^{2+\alpha_2}([0,1])$ for all $l$. 
Additionally, assume that $(g_0, h_0)$ satisfies the compatibility conditions of order $0$, in the sense that the function $(U_0, V_{1,0}, \ldots, V_{q,0})$  given by \eqref{eq:parametrization-gamma} and \eqref{eq:parametrization-chi} satisfies Definition \ref{def:compatibility_cond_order_zero}. 
Then, for any $T\in(0,1)$ and $(\bar{g}, \bar{h}) \in X^{T}_{(g_0,h_0)}$, there exists a unique solution $(g, h)\in X^{T}_{(g_0,h_0)}$ 
to the linear parabolic systems, \eqref{eq:higher-order-linear-g_l}$\sim$\eqref{eq:second-order-linear-h_l}, 
such that 
$g_{l}\in C^{\frac{4+\alpha_1}{4}, 4+\alpha_1}([0,T]\times [0,1])$, 
$h_{l} \in C^{\frac{2+\alpha_2}{2}, 2+\alpha_2}([0,T]\times [0,1])$, 
$\forall\, l$, and  
\begin{align}\label{estimate:(g,h)}
\nonumber
&\norm{(g, h)}_{ X^{T}_{(g_0, h_0)} } 
\leq  C_0\cdot \bigg{(} \sum^{q}_{l=1}
\norm{G(\partial_y^{3} \bar{g}_{l},
\ldots, \bar{g}_{l})}_{C^{\frac{\alpha_1}{4},\alpha_1}([0,T] \times [0,1])} 
\\ \nonumber
& ~~~~~~~~~~~~~~~~  
+\sigma^2 \sum^{q}_{l=1}\norm{ b_1(\partial_y \bar{h}_l, \bar{h}_l)}_{C^{\frac{\alpha_2}{2},\alpha_2}([0,T] \times [0,1])}
+ \frac{1}{\sigma^2} 
\sum_{\substack{l=1, \ldots, q 
\\  
y^{\ast}_l =2 \big\{ \frac{l}{2} \big\} }} 
\|\partial_y \bar{h}_l (\cdot, y^{\ast}_l)\|_{C^{\frac{1+\alpha_1}{4}}([0,T])}    
\\ \nonumber
&~~~~~~~~~~~~~~~~~~~~~~~~ +\sum_{\substack{(l,y^{\ast}) \in \{1, \ldots, q \} \times \{0,1\} \backslash \{(1,0)\}}} \| b_2(\partial_{y}^2\bar{g}_l, \partial_{y}\bar{g}_l, \bar{g}_l)(\cdot, y^{\ast})\|_{C^{\frac{1+\alpha_1}{4}}([0,T])}
\\ \nonumber
&~~~~~~~~~~~~~~~~~~~~~~~~
+|\varphi^{\prime}_0|+\sum^{q}_{l=0}| p_{l}| 
+\sum^{q}_{l=1}\Vert g_{l,0} \Vert_{C^{4+\alpha_1}([0,1])}+\sum^{q}_{l=1}\Vert h_{l,0} \Vert_{C^{2+\alpha_2}([0,1])} \\
&~~~~~~~~~~~~~~~~~~~~~~~~ + \text{either } 
\big{\|} b_1 (\partial_{y} \bar{g}_q, \bar{g}_q) \big(\cdot, 2 \big\{ \frac{q}{2} \big\} \big) \big{\|}_{C^{\frac{2+\alpha_1}{4}}([0,T])} ~~ \text{ or } ~~ |\varphi^{\prime}_q| \bigg{ )}
\end{align} 
where $C_0>0$ is independent of the non-homogeneous terms and the initial-boundary datum in the coupled linear parabolic system, 
\eqref{eq:higher-order-linear-g_l}$\sim$\eqref{eq:second-order-linear-h_l}. 
\end{theorem}

\begin{proof}[ The Proof of Theorem \ref{thm:STE_for_linear-fitting}]

The proof reduces to verifying that the linear parabolic systems \eqref{eq:higher-order-linear-g_l} and \eqref{eq:second-order-linear-h_l} satisfy the hypotheses of Solonnikov's theory 
\cite[Theorem 4.9, p. 121]{Solonnikov65}; see also \cite[Theorem 5.20]{LT26}. 
These conditions include parabolicity, smoothness of the coefficients, compatibility conditions, and the complementary conditions along the parabolic boundary. 
Within this framework, 
the estimate \eqref{estimate:(g,h)} is obtained by applying \cite[Theorem 4.9, p. 121]{Solonnikov65} with a suitable choice of data
for the linear problems associated with $g_l$ and $h_l$. Since the existence of solutions to 
\eqref{eq:second-order-linear-h_l} is established in \cite[Lemma 3.6]{LT26}, we omit the proof, and consequently only focus on 
establishing the existence of solutions to \eqref{eq:higher-order-linear-g_l}.

The present parabolic system differs from \emph{(2.19)$\sim$(2.26)} in \cite{LT26} (with $k=2$) only in the boundary conditions involving $g_1(t,0)$ and $g_q\!\left(t,2\bigl\{\frac{q}{2}\bigr\}\right)$. All remaining sufficient conditions for applying Solonnikov's theory, namely parabolicity, smoothness of the coefficients, and compatibility conditions, follow verbatim from \cite{LT26}. It therefore suffices to verify the complementing condition at the boundary for the parabolic system in \eqref{eq:higher-order-linear-g_l}.

We observe that the principal part of the parabolic system in \eqref{eq:higher-order-linear-g_l} is diagonal, coinciding with the operator $\partial_t+\partial_y^{4}$ acting on each component of $g$. Consequently, in the notation of Solonnikov in \cite{Solonnikov65}, we have $\mathcal{L}_0=\mathcal{L}$, and the associated characteristic polynomial is
\[
L(y,t,\mathrm{i}\xi,p)
=\det\mathcal{L}_0(y,t,\mathrm{i}\xi,p)
=(p+\xi^{4})^{mq}, ~~~~~~m=2N^2.
\]
Consequently,
\[
\widehat{\mathcal{L}}_0(y,t,\mathrm{i}\xi,p)
:= L\,\mathcal{L}_0^{-1}
= \mathrm{diag}\big(\widehat{\mathcal{L}}^{0}_{11},
\ldots,\widehat{\mathcal{L}}^{0}_{mq\,mq}\big),
\qquad
\widehat{\mathcal{L}}^{0}_{\ell\ell}
=(p+\xi^{4})^{mq-1},
\ \ell=1,\ldots,mq.
\]

\smallskip
\noindent\textbf{$\bullet$ The polynomial $M^{+}$.}

For $\operatorname{Re} p \ge 0$ with $p \neq 0$, the equation $p+(\mathrm{i}\xi)^{4}=0$ admits four roots, exactly two of which have positive real parts. We denote these roots by $\xi_1(y^\ast,p)$ and $\xi_2(y^\ast,p)$, and define the polynomial $M^{+}$ by
\[
M^{+}(y^\ast,\xi,p)
=(\xi-\xi_1(y^\ast,p))^{mq}(\xi-\xi_2(y^\ast,p))^{mq},
\]
in accordance with Solonnikov’s complementary condition. 

\smallskip 

{\bf $\bullet$ The complementary conditions along the boundaries $[0,T]\times \partial{I}$.} 

\smallskip 

The boundary conditions in \eqref{eq:higher-order-linear-g_l} can be rewritten as 
\begin{align*}
\mathcal{B}(y^{\ast}, \partial_{y}) 
g^{T}(t, y^{\ast})=\Phi^{\bar{g}, \bar{h}}(t, y^{\ast}), 
~~~y^{\ast} \in \{0,1\}. 
\end{align*} 
Here, the matrix $\mathcal{B}(y^{\ast}, \partial_{y})$ represents the simplified notation for the boundary operator 
$\mathcal{B}_0(y^{\ast}, t, \partial_{y}, \partial_{t})$ in Solonnikov's theory when $\mathcal{B}_0(y^{\ast}, t, \partial_{y}, \partial_{t})=\mathcal{B}_0(y^{\ast}, \partial_{y})$ is independent of $t$ and $\partial_t$, 
and is defined by 
\begin{align*}
\begin{cases}
(i) \text{ when $q$ is even:} 
\\ 
\mathcal{B}(y^\ast=0, \partial_{y})=
\text{diag}(B_0 , \overbrace{B_1, \ldots, B_1}^{\frac{q}{2}-1}, B_2), 
~~~
\mathcal{B}(y^\ast=1, \partial_{y})=\text{diag}(\overbrace{B_1, \ldots, B_1}^{\frac{q}{2}-1}, B_2 ), 
\\ 
\\ 
(ii) \text{ when $q$ is odd:}  
\\
\mathcal{B}(y^{\ast}=0, \partial_{y})=\text{diag}(B_0, \overbrace{B_1, \ldots, B_1}^{\frac{q-3}{2}}, B_2),
 ~~~~ 
\mathcal{B}( y^{\ast}=1, \partial_{y})=\text{diag}(\overbrace{B_1, \ldots, B_1}^{\frac{q-3}{2}}, B_2), 
\end{cases}
\end{align*}  
where   
\begin{align*}  
B_0 &:= 
\begin{pmatrix}
\mathrm{Id}_{m\times m}
\\ 
\partial_{y}\cdot \mathrm{Id}_{m\times m}
\\ 
\end{pmatrix}, 
~~~~~~~~~~~~~~~~~
B_1 := 
\begin{pmatrix} 
\mathrm{Id}_{m\times m} &- \mathrm{Id}_{m\times m}
\\ 
\partial_{y}\cdot \mathrm{Id}_{m\times m}&  \partial_{y} \cdot \mathrm{Id}_{m\times m}
\\ 
\partial_{y}^{2}\cdot \mathrm{Id}_{m\times m} & -\partial_{y}^{2} \cdot \mathrm{Id}_{m\times m}
\\ 
\partial_{y}^{3}\cdot \mathrm{Id}_{m\times m}& \partial_{y}^{3} \cdot \mathrm{Id}_{m\times m}
\\ 
\end{pmatrix}, 
\\ 
B_2 &:=
\begin{cases}
\begin{pmatrix}
\partial_{y} \cdot \mathrm{Id}_{m\times m}
\\ 
\partial_{y}^{3}\cdot \mathrm{Id}_{m\times m} 
\\ 
\end{pmatrix} ~~\text{ as the boundary condition is ~ } \partial_x g_q(t, 2\big\{ \frac{q}{2}\big\})=0, 
\\
\begin{pmatrix}
\partial_{y}^2 \cdot \mathrm{Id}_{m\times m}
\\ 
\partial_{y}^{3}\cdot \mathrm{Id}_{m\times m}
\\ 
\end{pmatrix} ~~\text{ as the boundary condition is ~ } \partial_y^2 g_q(t, 2\big\{ \frac{q}{2}\big\})=b_1(\partial_y \bar{g}_q, \bar{g}_q)(t, 2\big\{ \frac{q}{2}\big\}).
\end{cases}
\end{align*} 
Hence, each $\mathcal{B}(y^{\ast}, \partial_{y})$ is a diagonal block matrix. 
The term $\Phi^{\bar g,\bar h}$, defined on the right-hand side of \eqref{eq:higher-order-linear-g_l}, plays no role in the verification of the complementary conditions, and is only relevant in the a priori estimate \eqref{estimate:(g,h)}. 
The definition of $\Phi^{\bar g,\bar h}$ is analogous to that of
$\Phi^{\bar h}$ in \cite[(3.32)$\sim$(3.33)]{LT26}, and we therefore omit the details and refer the readers to \cite{LT26}.

To verify the complementary conditions of Solonnikov's theory, it suffices to show that the rows of 
\[
\mathcal{A}(y^\ast,t,\mathrm{i}\xi,p)
=\mathcal{B}(y^\ast,\mathrm{i}\xi)\,\hat{\mathcal{L}}_0(y^\ast,t,\mathrm{i}\xi,p)
\] 
are linearly independent modulo $M^{+}(y^\ast,\xi,p)$ for $\text{Re }p\ge0$, $p\neq0$. 
Since $\mathcal{B}$ is a diagonal block matrix and $\hat{\mathcal{L}}_0$ is diagonal, $\mathcal{A}$ is a diagonal block matrix. 
Hence, the verification of the complementary conditions on each block matrix is sufficient to ensure the complementary conditions of the matrix $\mathcal{A}$. 
More precisely, it is sufficient to verify that the rows of 
\[
\mathcal{A}_0
:= B_0  \widehat{\mathcal{L}}_{11}^{0}\, \mathrm{Id}_{m\times m},
\qquad
\mathcal{A}_1
:= B_1 \widehat{\mathcal{L}}_{11}^{0}  \begin{pmatrix}
\widehat{\mathcal{L}}_{11}^{0}\, \mathrm{Id}_{m\times m} \\
\widehat{\mathcal{L}}_{11}^{0}\, \mathrm{Id}_{m\times m}
\end{pmatrix},
\qquad
\mathcal{A}_2
:= B_2  \widehat{\mathcal{L}}_{11}^{0}\, \mathrm{Id}_{m\times m}
\] 
are linear independent modulo $M^{+}(y^\ast,\xi,p)$ 
for \(\mathrm{Re}\,p \ge 0\) and \(p \neq 0\).

Since the verification of linear independence for the matrix $\mathcal{A}_1$ has already been established in the case by letting $k=2$ in \cite{LT26}, it remains to verify those of the matrices $\mathcal{A}_2$ and $\mathcal{A}_0$, the cases associated with $B_2$ and $B_0$, respectively below.

\begin{description}

\item[Case $1^\circ$] When the matrix $B_2$ corresponds to the boundary condition $\partial_y g_q(t,0)$=0, we have   
\[
\mathcal{A}_2
:= B_2 \hat{\mathcal{L}}_{11}^0 \, \mathrm{Id}_{m\times m} 
=
\begin{pmatrix}
\hat{\mathcal{L}}_{11}^0 \, \mathrm{i}\xi \, \mathrm{Id}_{m\times m} 
\\[1mm]
\hat{\mathcal{L}}_{11}^0 \, (\mathrm{i}\xi)^3 \, \mathrm{Id}_{m\times m}
\end{pmatrix},
\]
evaluated at \(y^\ast=0\). The case corresponding to the boundary condition at $y^\ast=1$ is proved analogously.

To prove the linear independence of rows modulo \(M^{+}(0,\xi,p)\), we must show that
\begin{equation}\label{eq:comple-cond}
\omega\,\mathcal{A}_2 = 0 
\quad \text{mod } M^{+}(0,\xi,p)
\quad\Longrightarrow\quad \omega=0 .
\end{equation}
From a straightforward computation, 
the algebraic equation \eqref{eq:comple-cond} is equivalent to
\[
\bigl( (\mathrm{i}\xi) \, \omega_j + (\mathrm{i} \xi)^3 \omega_{m+j} \bigr)
(p+\xi^4)^{mq-1}=0
\quad \text{mod } M^{+}(0,\xi,p).
\] 
Dividing by the factors corresponding to the roots of \(M^+\) yields  the reduced condition
\begin{equation}\label{eq:reduced}
(\mathrm{i} \xi) \, \omega_j + (\mathrm{i}\xi)^3 \omega_{m+j}
=0
\quad \text{mod } (\xi-\xi_1)(\xi-\xi_2),
\end{equation}
since the remaining polynomial factor is co-prime to 
\((\xi-\xi_1)(\xi-\xi_2)\). 

Substituting $\xi=\xi_1$ and $\xi=\xi_2$ into \eqref{eq:reduced} yields the linear system
\[
(\omega_j,\omega_{m+j}) C_1:=
(\omega_j,\omega_{m+j})
\begin{pmatrix}
 \mathrm{i} \xi_1  & \mathrm{i} \xi_2 \\[1mm]
(\mathrm{i} \xi_1)^3  & (\mathrm{i} \xi_2)^3
\end{pmatrix}
=0.
\]
Since  
\[
\det C_1
= \mathrm{i} \,\xi_1  \xi_2  \bigl( \xi_2^{2} - \xi_1^{2} \bigr)
\neq 0, 
\]
it follows that the matrix \(C_1\) is invertible, and 
$(\omega_j,\omega_{m+j})=(0,0)$, for all $j$.
Therefore, we conclude $\omega=0$. 
This proves the linear independence for the rows of 
$\mathcal{A}_2$ in the corresponding boundary conditions.

\item[Case $2^\circ$] 
When the matrix $B_2$ corresponds to the boundary conditions, 
\[
\partial_y^2 g_q\!\left(t,\,2\bigl\{\tfrac{q}{2}\bigr\}\right)
= b_1(\partial_y \bar g_q,\bar g_q)\!\left(t,\,2\bigl\{\tfrac{q}{2}\bigr\}\right), 
\] 
we have
\[
\mathcal{A}_2
:= B_2\, \hat{\mathcal{L}}_{11}^0\, \mathrm{Id}_{m\times m}=\begin{pmatrix}
\hat{\mathcal{L}}_{11}^0 (\mathrm{i}\xi)^2\, \mathrm{Id}_{m\times m} \\[1mm]
\hat{\mathcal{L}}_{11}^0 (\mathrm{i}\xi)^3\, \mathrm{Id}_{m\times m}
\end{pmatrix}
\]
evaluated at $y^\ast = 2\bigl\{\tfrac{q}{2}\bigr\}$.

To verify linear independence of the rows modulo $M^{+}(2\bigl\{\tfrac{q}{2}\bigr\},\xi,p)$,
we must prove
\begin{equation}\label{eq:comple-cond-2}
\omega\,\mathcal{A}_2 = 0
\quad \text{mod } M^{+}(2\bigl\{\tfrac{q}{2}\bigr\},\xi,p)
\quad\Longrightarrow\quad \omega=0 .
\end{equation}

From a straightforward computation, 
the algebraic equation \eqref{eq:comple-cond-2} is equivalent to
\[
\bigl( (\mathrm{i}\xi)^2 \omega_j + (\mathrm{i}\xi)^3 \omega_{m+j} \bigr)
(p+\xi^4)^{mq-1} = 0
\quad \text{mod } M^{+}(2\bigl\{\tfrac{q}{2}\bigr\},\xi,p).
\]
Dividing the factors corresponding to the roots of $M^+$ yields
\begin{equation}\label{eq:reduced-2}
(\mathrm{i}\xi)^2 \omega_j + (\mathrm{i}\xi)^3 \omega_{m+j}
=0
\quad \text{mod } (\xi-\xi_1)(\xi-\xi_2),
\end{equation}
since the remaining factor is co-prime to $(\xi-\xi_1)(\xi-\xi_2)$.

Substituting $\xi=\xi_1$ and $\xi=\xi_2$ into 
\eqref{eq:reduced-2} yields 
\[
(\omega_j,\omega_{m+j}) C_2 :=
(\omega_j,\omega_{m+j})
\begin{pmatrix}
(\mathrm{i}\xi_1)^2 & (\mathrm{i}\xi_2)^2 \\[1mm]
(\mathrm{i}\xi_1)^3 & (\mathrm{i}\xi_2)^3
\end{pmatrix}
=0.
\]
Since 
\[
\det C_2
= \mathrm{i}\,\xi_1^{2}\xi_2^{2}\bigl(\xi_2^{2}-\xi_1^{2}\bigr)\neq 0,
\]
it follows that the matrix $C_2$ is invertible. Hence $(\omega_j,\omega_{m+j})=(0,0)$ for all $j$, and thus $\omega=0$.
This proves the linear independence for the rows of  $\mathcal{A}_2$ in the corresponding boundary conditions.

\item[Case $3^\circ$] 
When the matrix $B_0$ corresponds to the boundary condition, 
\[
\partial_y g_1(t,0)=0, 
\]
we have 
\[
\mathcal{A}_0
:= B_0\, \hat{\mathcal{L}}_{11}^0\, \mathrm{Id}_{m\times m}
=
\begin{pmatrix}
\hat{\mathcal{L}}_{11}^0\, \mathrm{Id}_{m\times m} \\[1mm]
\hat{\mathcal{L}}_{11}^0\, (\mathrm{i}\xi)\, \mathrm{Id}_{m\times m}
\end{pmatrix}.
\]
To show that the rows are linearly independent modulo $M^{+}(0,\xi,p)$,
we must verify that
\begin{equation}\label{eq:comple-cond-3}
\omega\,\mathcal{A}_0 = 0
\quad \text{mod } M^{+}(0,\xi,p)
\quad \Longrightarrow \quad \omega=0 .
\end{equation}
From a straightforward computation, 
the algebraic equation \eqref{eq:comple-cond-3} is equivalent to
\[
\bigl( \omega_j + (\mathrm{i}\xi)\omega_{m+j} \bigr)
(p+\xi^4)^{mq-1}
=0
\quad \text{mod } M^{+}(0,\xi,p).
\]
Dividing the factors corresponding to the roots of $M^+$ yields
\begin{equation}\label{eq:reduced-3}
\omega_j + (\mathrm{i}\xi)\omega_{m+j}
=0
\quad \text{mod } (\xi-\xi_1)(\xi-\xi_2),
\end{equation}
since the remaining factor is co-prime to $(\xi-\xi_1)(\xi-\xi_2)$.

Substituting $\xi=\xi_1$ and $\xi=\xi_2$ into  \eqref{eq:reduced-3} at $\xi=\xi_1$ and $\xi=\xi_2$ yields 
\[
(\omega_j,\omega_{m+j})\, C_3 :=
(\omega_j,\omega_{m+j})
\begin{pmatrix}
1 & 1 \\[1mm]
\mathrm{i}\xi_1 & \mathrm{i}\xi_2
\end{pmatrix}
=0.
\]

Since 
\[
\det C_3
= \mathrm{i}(\xi_2 - \xi_1) \neq 0,
\]
it follows that the matrix $C_3$ is invertible. Therefore
\((\omega_j,\omega_{m+j})=(0,0)\) for all $j$, which implies $\omega=0$. 
This proves the linear independence for the rows of  $\mathcal{A}_0$ in the corresponding boundary conditions. 

\end{description}

The verification of the complementary conditions at the boundaries is now completed.

This verifies all conditions required for Solonnikov's theory, establishing the existence of solutions to \eqref{eq:higher-order-linear-g_l} and \eqref{eq:second-order-linear-h_l} and completing the proof; see \cite[Lemma 3.6]{LT26} for details. 

\end{proof}

\begin{theorem}\label{theorem-short-time}
Suppose that all the assumptions in Theorem \ref{thm:STE_for_linear-fitting} are satisfied.
Then, the following statements hold:   
\begin{itemize}
\item[(i)] 
There exists a positive number 
$t_0=t_0\left(N, \alpha_1, \alpha_2, \sigma, g_0, h_0\right)>0$ 
such that $(g, h)\in X^{t_0}_{(g_0,h_0)}$ 
is the unique solution to \eqref{eq:nonlinear-g-h_l}.
\item[(ii)]  
The solution $(g,h)$ satisfies the regularity 
\begin{equation*} 
\begin{cases}
g_l  \in C^{\frac{4+\alpha_1}{4}, 4+\alpha_1}\left([0,t_0] \times [0,1]\right)\bigcap 
C^{\infty}\left((0,t_0]\times [0,1]\right), 
~~~~~~~~~ l\in\{1, \ldots, q\}, 
\\
h_l  \in C^{\frac{2+\alpha_2}{2},2+\alpha_2}\left([0,t_0] \times [0,1]\right)\bigcap 
C^{\infty}\left((0,t_0]\times [0,1]\right), 
~~~~~~~~~ l\in\{1, \ldots, q\}.  
\end{cases}
\end{equation*} 
\item[(iii)] 
Moreover, the induced map $U: [x_0, x_q]\rightarrow \mathbf{SU}(N) \subset \mathbb{R}^{2N^2}$,
defined by \eqref{eq:parametrization-gamma} 
satisfies  
$U(t,\cdot) \in C^{2}([x_0, x_q]), ~ \forall\, t \in [0,t_0]$. 
\end{itemize}  
\end{theorem}

\begin{remark}
The constant $t_0>0$ in 
Theorem~\ref{theorem-short-time} depends only on $m$, $\alpha_1$, $\alpha_2$, $\sigma$, and the certain norm of the initial data $(g_0,h_0)$. 
We refer to \cite[Remark~3.10]{LT26} for the detailed discussion of the dependence.
\end{remark}

\begin{proof}[The Proof of Theorem \ref{theorem-short-time}]

We focus on the case with the boundary condition 
\[
\partial_{y}^2 g_{q}(t, 2 \big\{ \frac{q}{2}\big\})= b_1( \partial_y g_q, g_q)(t, 2 \big\{ \frac{q}{2}\big\}),
\] 
for which we provide a detailed proof. 
The remaining case with boundary condition 
\[
\partial_{y}g_{q}(t,2 \big\{ \frac{q}{2}\big\})=(-1)^{q-1}\varphi^{\prime}_q
\] 
follows similarly and is, in fact, simpler, requiring only the replacement of 
\[
\big{\|} b_1 (\partial_{y} \bar{g}_q, \bar{g}_q) \big(\cdot, 2 \big\{ \frac{q}{2} \big\} \big) \big{\|}_{C^{\frac{2+\alpha_1}{4}}([0,T])}
\] 
with $|\varphi^{\prime}_q|$ in the estimates below. 

The proof is divided into several steps.

\begin{itemize}

\item[\bf (i).] {\bf Existence of Local Solutions}: 
The existence of solution to \eqref{eq:nonlinear-g-h_l} in the space $X^{t_0}_{(g_0,h_0)}$.

\bigskip 


{\bf $\bullet$ Self-maps, i.e., $\exists$ $t_1 >0$ so that 
$\mathcal{G}\left( X^{T}_{(g_0,h_0)} \cap B_{K_1}\right) \subset X^{T}_{(g_0,h_0)} \cap B_{K_1}$, 
$\forall\, T\in(0, t_1)$, and some $K_1\in(0,\infty)$.} 

Applying Theorem~\ref{thm:STE_for_linear-fitting} to the linear parabolic systems \eqref{eq:higher-order-linear-g_l} and \eqref{eq:second-order-linear-h_l}, we define the operator
\begin{align*} 
\nonumber
\mathcal{G}:  X^{T}_{(g_0,h_0)} \cap B_K 
&\to 
X^{T}_{(g_0,h_0)} \cap B_K 
\\ 
(\bar{g}, \bar{h})
&\mapsto  (g,h)
\text{,} 
\end{align*}
where $(g,h)$ denotes the unique solution of
\eqref{eq:higher-order-linear-g_l}$\sim$\eqref{eq:second-order-linear-h_l} for sufficiently small $T\in(0,1)$.  
We demonstrate that $\mathcal{G}$ acts as a self-map on $B_{K_1}$ for some $K_1\in(0,\infty)$, following the same argument as in \cite[Theorem~3.10]{LT26}. Specifically, interpolation inequalities are applied to the estimate \eqref{estimate:(g,h)}, where each term not involving the initial data $g_{l,0}$ or $h_{l,0}$  is rewritten by adding and subtracting these quantities, expressing it in terms of either 
$\bar g_l-g_{l,0}$ or $\bar h_l-h_{l,0}$. 
We then apply the triangle inequality together with the algebraic identity
\begin{equation}\label{identity-a^n-b^n}
a^p-b^p=(a-b)\bigl(a^{p-1}+a^{p-2}b+\cdots+ab^{p-2}+b^{p-1}\bigr),
\qquad p\in\{2,3,\ldots\},
\end{equation}
and the technical lemmas \cite[Remark~B.1, Lemma~B.2]{DLP21} and \cite[Lemma~5.10]{LT26},
with $v=\bar g_l-g_{l,0}$ or $v=\bar h_l-h_{l,0}$.

As a consequence, we obtain the estimate
\[
\|(g,h)\|_{X^{T}_{(g_0,h_0)}}
\le C_1\,T^{\beta}
+ C(N,\sigma,K), 
\] 
for some $\beta>0$. 
Choosing $T\in(0,1)$ sufficiently small ensures that $\mathcal{G}$ maps $X^{T}_{(g_0,h_0)}$ into itself, and hence $\mathcal{G}$ is a self-map. 
This follows by applying the triangle inequality \eqref{identity-a^n-b^n} together with the technical lemmas \cite[Remark~B.1, Lemma~B.2]{DLP21} and \cite[Lemma~5.10]{LT26}, with $v=\bar g_l-g_{l,0}$ or $v=\bar h_l-h_{l,0}$, yielding 
\begin{align*} 
&\|b_1(\partial_y\bar{h}_l, \bar{h}_l)- b_1(\partial_y h_{l,0}, h_{l,0})\|_{ C^{\frac{\alpha_2}{2}, \alpha_2} ([0,T] \times [0,1])} 
\leq C(N, K)\cdot 
T^{\beta_1}, 
\\ 
& \sum\limits^{q}_{l=1}\Vert G(\partial_y^{3} \bar{g}_{l}, \ldots, \bar{g}_{l})
-G(\partial_y^{3} g_{l,0}, \ldots, g_{l,0}) 
\Vert_{ C^{\frac{\alpha_1}{4}, \alpha_1}([0,T] \times [0,1])}  
\leq  C\left(N, K \right) \cdot T^{\beta_1} 
\text{,}
\\ 
&\| \partial_y\bar{h}_l (\cdot, y^{\ast}_l)- \partial_y h_{l,0} (y^{\ast}_l) \|_{C^{\frac{1+\alpha_1}{4} }([0,T])} 
\leq 4 K T^{\beta_3},
\\ \nonumber 
&\sum_{\substack{(l,y^{\ast}) \in \{1, \ldots, q \} \times \{0,1\} \backslash \{(1,0)\}}}\| b_2(\partial_y^2 \bar{g}_l, \partial_y \bar{g}_l,  \bar{g}_l)(\cdot, y^{\ast}) -b_2(\partial_y^2 g_{l,0}, \partial_y g_{l,0},  g_{l,0})(y^{\ast}) \|_{C^{\frac{1+\alpha_1}{4}}([0,T])} 
\\ 
& ~~~~~~~~~~~~~~~~~~~~~~~~~~~~~~~~~~~~~~~~~~~~~~~~~~~~~~ \leq C\left(N, K \right) \cdot T^{\beta_1}, 
\\ \nonumber 
&
\bigg{\|} b_1(\partial_y \bar{g}_q,  \bar{g}_q)\big(\cdot, 2 \big\{ \frac{q}{2} \big\} \big) -b_1(\partial_y g_{q,0},  g_{q,0})\big(2 \big\{ \frac{q}{2} \big\} \big) \bigg{\|}_{C^{\frac{2+\alpha_1}{4}}([0,T])}\leq  C\left(N, K \right)  \cdot T^{\beta_1}.
\end{align*}
where 
$\beta_1=\min \big{ \{ } \frac{\alpha_1}{4},\frac{1-\alpha_1}{4} \big{ \} }$, $\beta_2=\min \big{ \{ } \frac{\alpha_2}{2}, \frac{1-\alpha_2}{2} \big{ \} }$, and $\beta_3=\min \big{ \{ } \frac{\alpha_2}{2}, \frac{1+\alpha_2}{2}-\frac{1+\alpha_1}{4} \big{ \} }$. 
Note that the simplified notation $C(N,K)$ used above reflects the bounds  
\begin{equation*} 
\begin{cases}
\| g_{l,0}\|_{C^{4+\alpha}([0,1])} \leq \Vert g_l 
\Vert_{ C^{\frac{4+\alpha}{4}, 4+\alpha}([0,T] \times [0,1])}\le K, 
\\ 
\| h_{l,0}\|_{C^{4+\alpha}([0,1])} \leq \Vert h_l 
\Vert_{ C^{\frac{4+\alpha}{4}, 4+\alpha}([0,T] \times [0,1])}\le K.   
\end{cases}
\end{equation*}

By the inequality in Theorem \ref{thm:STE_for_linear-fitting} and the triangle inequality in H\"{o}lder spaces, we obtain 
\begin{align} 
\label{eq:|(g,h)|_{X^T}}
&\norm{(g,h)}_{X^{T}_{(g_0,h_0)}} 
\\  \nonumber 
&\leq  C_1 T^{\beta} 
+C_0 \cdot  
\bigg( 
\sum^{q}_{l=1}\norm{  G(\partial_y^{3} g_{l,0}, \ldots, g_{l,0})}_{ C^{\alpha_1} ([0,1]) } 
 + \sigma^2 \sum^{q}_{l=1}\norm{ b_1(\partial_y h_{l,0}, h_{l,0})}_{C^{\alpha_2}( [0,1])} 
\\ \nonumber 
& +\sum_{\substack{(l,y^{\ast}) \in \{1, \ldots, q \} \times \{0,1\} \backslash \{(1,0)\}}} |b_2(\partial_{y}^2g_{l,0}, \partial_{y}g_{l,0}, g_{l,0})( y^{\ast})| + \big(1+\frac{1}{\sigma^2} \big) \sum^{q}_{l=1}\Vert h_{l,0} \Vert_{ C^{2+\alpha_2} ([0,1])} \\ \nonumber 
&~~~~~~~~+ \big|b_1(\partial_y g_{q,0},  g_{q,0})\big(2 \big\{ \frac{q}{2} \big\} \big)  \big| 
+|\varphi^{\prime}_0|+ \sum\limits_{l=0}^{q} |p_{l}| 
+\sum^{q}_{l=1}
\Vert g_{l,0}\Vert_{C^{4+\alpha_1}([0,1])}  \bigg) 
\end{align} 
where  
$\beta:=\min\{\beta_1, \beta_2, \beta_3\}$, 
$T\in(0,1)$, $C_1 =C_1 \left(N, \sigma, K\right)$.  
By letting 
\begin{align*} 
& K=K_1:=2 C_0 \cdot  
\bigg( 
\sum^{q}_{l=1}\norm{G(\partial_y^{3} g_{l,0}, \ldots, g_{l,0})}_{ C^{\alpha_1} ([0,1]) } 
 + \sigma^2 \sum^{q}_{l=1}\norm{ b_1(\partial_y h_{l,0}, h_{l,0})}_{C^{\alpha_2}( [0,1])} 
\\ \nonumber 
& +\sum_{\substack{(l,y^{\ast}) \in \{1, \ldots, q \} \times \{0,1\} \backslash \{(1,0)\}}} |b_2(\partial_{y}^2g_{l,0}, \partial_{y}g_{l,0}, g_{l,0})(y^{\ast})|+ \big(1+\frac{1}{\sigma^2} \big) \sum^{q}_{l=1}\Vert h_{l,0} \Vert_{ C^{2+\alpha_2} ([0,1])}  
\\ \nonumber 
& ~~~~~~~~+  \big|b_1(\partial_y g_{q,0},  g_{q,0})\big(2 \big\{ \frac{q}{2} \big\} \big)  \big| 
+|\varphi^{\prime}_0| + \sum\limits_{l=0}^{q} |p_{l}| 
+\sum^{q}_{l=1}
\Vert g_{l,0}\Vert_{C^{4+\alpha_1}([0,1])} \bigg) 
\end{align*} 
above and by choosing $T=t_1>0$ in \eqref{eq:|(g,h)|_{X^T}}, where $t_1$ satisfies  
$ C_1 t_1^{\beta} \leq \frac{K_1}{2}$, 
we obtain  
$\norm{(g,h)}_{X^{t_1}_{(g_0,h_0)}}\leq K_1$.
Thus, we conclude 
\begin{align}
\label{eq:self_map}
\mathcal{G}\left(X^{T}_{(g_0,h_0)} 
\cap B_{K_1}\right) \subset X^{T}_{(g_0,h_0)} \cap B_{K_1}, 
~~~~~~\forall\, T \in (0,t_1]. 
\end{align}

\bigskip


{\bf $\bullet$ Contraction maps.}

We now establish that $\mathcal{G}$ is a contraction on $X^{T}_{(g_0,h_0)} \cap B_{K_1}$ for sufficiently small $T>0$. Specifically, there exists $t_0\in(0,t_1]$ such that for any $T\in(0,t_0]$ and any 
$$
(\bar{g}, \bar{h}), (\bar{e}, \bar{f})\in X^{T}_{(g_0,h_0)} \cap B_{K_1}, 
$$ 
the following estimate holds:
\begin{align}
\label{est:contraction-map-gamma+chi}
\norm{(g,h)-(e,f)}_{X^{T}_{(g_0,h_0)}} 
\leq C_{2} \, T^{\beta} \,  \norm{(\bar{g}, \bar{h})-(\bar{e}, \bar{f})}_{X^{T}_{(g_0,h_0)}} 
<\norm{(\bar{g}, \bar{h})-(\bar{e}, \bar{f})}_{X^{T}_{g_0}},  
\end{align} 
where 
$(g,h)=\mathcal{G} \big( (\bar{g}, \bar{h}) \big)$, $(e,f)=\mathcal{G} \big( (\bar{e}, \bar{f}) \big)$, 
and the exponent 
\begin{align}
\label{def:beta}
\beta=\min\{\beta_1, \beta_2, \beta_3\}.
\end{align}  
The constant $C_2 = C_2 \left(N, \sigma, 
K_1\right)$ is positive and depends only on the specified data.

Observe that, from 
\eqref{eq:higher-order-linear-g_l} and 
\eqref{eq:second-order-linear-h_l}, 
the difference $g-e$ satisfies the linear parabolic system:  
\begin{equation}
\label{eq:linear-g-e}
\begin{cases} 
\partial_t ( g_{l}-e_{l})+\partial_y^{4}( g_{l}-e_{l})=
G(\partial_y^{3} \bar{g}_{l}, \ldots, \bar{g}_{l})-G(\partial_y^{3} \bar{e}_{l}, \ldots, \bar{e}_{l}), 
\\
~~~~~~~~~~~~~~~~~~~~~~~~~~~~~~~~~~~~~~~~~~~~~~~~~~~~~~~~~~~~ 
\text{ in } (0,T) \times (0,1), ~ l \in\{1, \ldots, q\}, 
\\
(g_{l}-e_{l})(0,y)= 0, 
~~~~~~~~~~~~~~~~~~~~~~~~~~~~~~~~~~~~~~~~~~~~~~~~~~~~~~~~~~~~~ 
l \in \{1, \ldots, q\}, 
\\
 \partial_y(g_{1}-e_{1})(t, 0)=0,  
\\
(g_{l}-e_{l})(t, y^{\ast}_{l})= (g_{l+1}-e_{l+1})(t, y^{\ast}_{l}),   
~~~~~~~~~~~~~~~~~~~ 
l \in \{1, \ldots, q-1\}, 
y^{\ast}_{l}=2\{\frac{l}{2}\}, 
\\
\partial_{y}^{\mu}\left(g_{l}-e_{l}\right) (t,y^{\ast}_{l})
+(-1)^{\mu-1}\partial_{y}^{\mu} \left(g_{l+1}-e_{l+1}\right)(t,y^{\ast}_{l})=0,  
\\
~~~~~~~~~~~~~~~~~~~~~~~~~~~~~~~~~~~~~~~~~~~~~~~~~
\mu \in \{1, 2\}, 
~ l \in \{1, \ldots, q-1\}, ~
y^{\ast}_{l}=2\{\frac{l}{2}\},  
\\ 
\partial_{y}^{3}(g_{l}-e_{l})(t,y^{\ast}_l)+\partial_{y}^{3}( g_{l+1}-e_{l+1})(t,y^{\ast}_{l})= \frac{1}{\sigma^2}  \partial_{y} (\bar{h}_{l}-\bar{f}_{l})(t,y^{\ast}_{l}) 
\\
~~~~~~~~~~~~~~~~~~~~~~~~~ 
+b_2(\partial_{y}^2 \bar{g}_l, \partial_{y}\bar{g}_l, \bar{g}_l)(t, y^{\ast}_{l})-b_2(\partial_{y}^2 \bar{e}_l, \partial_{y}\bar{e}_l, \bar{e}_l)(t, y^{\ast}_{l}) 
\\
~~~~~~~~~~~~~~~~~~~~~~~~~ 
-b_2(\partial_{y}^2 \bar{g}_{l+1}, \partial_{y}\bar{g}_{l+1}, \bar{g}_{l+1})(t,y^{\ast}_{l})+b_2(\partial_{y}^2 \bar{e}_{l+1}, \partial_{y}\bar{e}_{l+1}, \bar{e}_{l+1})(t,y^{\ast}_{l}), 
\\
~~~~~~~~~~~~~~~~~~~~~~~~~~~~~~~~~~~~~~~~~~~~~~~~~~~~~~~~~~~~~~~~~~ 
l \in \{1, \ldots, q-1\}, 
y^{\ast}_{l}= 2\{\frac{l}{2}\}, 
\\  
\partial_y^2 (g_{q}-e_{q}) (t,2 \big\{ \frac{q}{2}\big\})
=b_1(\partial_y \bar{g}_q, \bar{g}_q)(t,2 \big\{ \frac{q}{2}\big\})-b_1(\partial_y \bar{e}_q, \bar{e}_q)(t,2 \big\{ \frac{q}{2}\big\}) 
\\ 
~~~~~ \text{ or } ~~~~~
\\ ~~~
(g_{q}-e_{q})(t,2 \big\{ \frac{q}{2}\big\})=0, 
\\
\partial_{y}^{3}(g_{q}-e_{q})(t,2 \big\{ \frac{q}{2}\big\}))= \frac{1}{\sigma^2}  \partial_{y} (\bar{h}_{q}-\bar{f}_{q})(t,2 \big\{ \frac{q}{2}\big\}) \\
~~~~~~~~~~~~~~~~~~~~~~~~~~~~~~~~~~~~
+b_2(\partial_{y}^2 \bar{g}_q, \partial_{y}\bar{g}_q, \bar{g}_q)(t, y^{\ast}_{l})-b_2(\partial_{y}^2 \bar{e}_q, \partial_{y}\bar{e}_q, \bar{e}_q)(t,2 \big\{ \frac{q}{2}\big\}),
\end{cases}
\end{equation} 
and the difference $h-f$ fulfills the linear parabolic system:   
\begin{equation}
\label{eq:linear-h-f}
\begin{cases} 
\partial_t ( h_{l}-f_{l})- \sigma^2  \partial_y^{2}( h_{l}-f_{l})
= \sigma^2 b_1(\partial_y \bar{h}_l, \bar{h}_l)
- \sigma^2 b_1(\partial_y \bar{f}_l, \bar{f}_l),  
\\
~~~~~~~~~~~~~~~~~~~~~~~~~~~~~~~~~~~~~~~~~~~~~~~~ 
\text{ in } (0,T) \times (0,1), l \in\{1, \ldots, q\}, 
\\
(h_{l}-f_{l})(0,y)= 0,  
~~~~~~~~~~~~~~~~~~~~~~~~~~~~~~~~~~~~~~~~~~~~~~~ 
l \in \{1, \ldots, q\}, 
\\
(h_{l}-f_{l})(t, y^{\ast}_l)= 0, 
~~~~~~~~~~~~~~~~~~~~~~~~~~~~~ 
l \in \{1, \ldots, q\}, 
~ y^{\ast}_l=2\{\frac{l+1}{2}\}, 
\\
(h_{l}-f_{l})(t, y^{\ast}_l)=(g_{l}-e_{l})(t, y^{\ast}_l),  
~~~~~~~~~~~~~~~ 
l \in \{1, \ldots, q\}, ~
y^{\ast}_l= 2\{\frac{l}{2}\},
\end{cases}
\end{equation} 
where $t \in [0,T]$ and $y \in [0,1]$.

Let $(g,h)=\mathcal{G}(\bar g,\bar h)$ and $(e,f)=\mathcal{G}(\bar e,\bar f)$.
Since the initial data satisfy
\[
\bar g_l(0,\cdot)=\bar e_l(0,\cdot)=g_{l,0}(\cdot),
\qquad
\bar h_l(0,\cdot)=\bar f_l(0,\cdot)=h_{l,0}(\cdot),
\]
and the compatibility conditions of order $0$ for \eqref{eq:higher-order-linear-g_l} and \eqref{eq:second-order-linear-h_l},  
the differences $g-e$ and $h-f$ 
have vanishing initial data and satisfy the compatibility conditions of order $0$ for \eqref{eq:linear-g-e} and \eqref{eq:linear-h-f}, respectively. 
It follows from Theorem 
\ref{thm:STE_for_linear-fitting} that $g-e$ and $h-f$ are the unique solutions to \eqref{eq:linear-g-e} and \eqref{eq:linear-h-f}, respectively. 
Moreover, the following estimate holds:
\begin{align}
\label{est:(g,h)-(e,f)}
&\norm{(g,h)-(e,f)}_{X^{T}_{(g_0,h_0)}} 
\\ \nonumber 
&\leq C_0 \cdot \bigg(\sum\limits^{q}_{l=1} 
\norm{ G(\partial_y^{3} \bar{g}_{l}, \ldots, \bar{g}_{l})-G(\partial_y^{3} \bar{e}_{l}, \ldots, \bar{e}_{l})}_{C^{\frac{\alpha_1}{4}, \alpha_1}([0,T] \times [0,1])}
\\ \nonumber 
&~~~~~~~~+\sigma^2 \sum\limits^{q}_{l=1} 
\norm{ b_1(\partial_y \bar{h}_l, \bar{h}_l)-b_1(\partial_y \bar{f}_l, \bar{f}_l) }_{C^{\frac{\alpha_2}{2}, \alpha_2}([0,T] \times [0,1])}
\\ \nonumber 
&~~~~~~~~+\frac{1}{\sigma^2}  \sum_{\substack{l=1, \ldots, q 
\\ y^{\ast}_l =2\{\frac{l}{2}\} }} 
\|\partial_y \bar{h}_l(\cdot, y^{\ast}_l) -\partial_y\bar{f}_l(\cdot, y^{\ast}_l)\|_{C^{\frac{1+\alpha_1}{4}}([0,T])} 
\\ \nonumber
&~~~~~~~~+ \sum_{\substack{(l,y^{\ast}) \in \{1, \ldots, q \} \times \{0,1\} \backslash \{(1,0)\}}} \| b_2(\partial_{y}^2\bar{g}_l, \partial_{y}\bar{g}_l, \bar{g}_l)(\cdot, y^{\ast})-b_2(\partial_{y}^2\bar{e}_{l}, \partial_{y}\bar{e}_{l}, \bar{e}_{l})(\cdot, y^{\ast})\|_{C^{\frac{1+\alpha_1}{4}}([0,T])} 
\\ \nonumber 
&~~~~~~~~+\big{\|} b_1 \big(\partial_{y} \bar{g}_q, \bar{g}_q)(\cdot, 2 \{ \frac{q}{2}\})-b_1 \big(\partial_{y} \bar{f}_q, \bar{f}_q) \big(\cdot, 2 \big\{ \frac{q}{2} \big\} \big)  \big{\|}_{C^{\frac{2+\alpha_1}{4}}([0,T])}  \bigg)
\text{.}
\end{align} 
Applying the algebraic identity 
\eqref{identity-a^n-b^n}, the triangle inequality in 
H\"{o}lder spaces, and the technical estimates from \cite[Remark~B.1, Lemma~B.2]{DLP21} and \cite[Lemma~5.10]{LT26} with $v=\bar g_l-\bar e_l$ or $v=\bar h_l-\bar f_l$, we obtain the following estimates:
\begin{align*}
& \|b_1(\partial_y \bar{h}_l, \bar{h}_l)- b_1(\partial_y\bar{f}_l, \bar{f}_l)\|_{ C^{\frac{\alpha_1}{4}, \alpha_1}([0,T] \times [0,1])}  \leq  
C\left(N, K_1\right) \cdot T^{\beta_1} \cdot \| \bar{h}_l-\bar{f}_l \|_{ C^{\frac{4+\alpha_1}{4}, 4+\alpha_1}([0,T]\times [0,1])}\\ \nonumber
&\norm{ G(\partial_y^{3} \bar{g}_{l}, \ldots, \bar{g}_{l})-G(\partial_y^{3} \bar{e}_{l}, \ldots, \bar{e}_{l})}_{C^{\frac{\alpha_1}{4}, \alpha_1}([0,T] \times [0,1])} \\
&~~~~~~~~~~~~~~~~~~~~~~~~~~~~~~~~~~~~~~~~~~~~~ \leq   C(N, K_1 )
\cdot 
\|\bar{g}_l- \bar{e}_l\|_{ C^{\frac{4+\alpha_1}{4}, 4+\alpha_1}([0,T] \times [0,1])} 
\cdot T^{\beta_1}, 
\\
\nonumber
&\|\partial_y \bar{h}_l(\cdot, y^{\ast}_l) - \partial_y\bar{f}_l(\cdot, y^{\ast}_l)\|_{C^{\frac{1+\alpha_1}{4}}([0,T])} 
\leq 2 T^{\beta_3} \|\bar{h}_l -\bar{f}_l \|_{C^{\frac{2+\alpha_2}{2}, 2+\alpha_2}([0,T] \times [0,1])} 
\\ 
\nonumber
&\| b_2(\partial_{y}^2\bar{g}_l, \partial_{y}\bar{g}_l, \bar{g}_l)(\cdot, y^{\ast})-b_2(\partial_{y}^2\bar{e}_{l}, \partial_{y}\bar{e}_{l}, \bar{e}_{l})(\cdot, y^{\ast})\|_{C^{\frac{1+\alpha_1}{4}}([0,T])}  
\\
&~~~~~ \leq  C(N, K_1 ) T^{\beta_2} \|\bar{g}_l -\bar{e}_l \|_{C^{\frac{4+\alpha_1}{4}, 4+\alpha_1}([0,T] \times [0,1])}, ~~~ \forall\, (l,y^{\ast}) \in \{1, \ldots, q \} \times \{0,1\} \backslash \{(1,0)\}, 
\\
& 
\big{\|} b_1 \big(\partial_{y} \bar{g}_q, \bar{g}_q)(\cdot, 2 \{ \frac{q}{2}\})-b_1 \big(\partial_{y} \bar{f}_q, \bar{f}_q) \big(\cdot, 2 \big\{ \frac{q}{2} \big\} \big)  \big{\|}_{C^{\frac{2+\alpha_1}{4}}([0,T])}
\\
& ~~~~~~~~~~~~~~~~~~~~~~~  ~~~~~~~~~~~~~~~~~~~~~~~~~~~~~~~~~\leq C(N, K_1 ) T^{\beta_2} \|\bar{g}_q -\bar{e}_q \|_{C^{\frac{4+\alpha_1}{4}, 4+\alpha_1}([0,T] \times [0,1])} 
\end{align*}
where $C(N, K_1)>0$ is constant.  
Together with 
\eqref{def:beta} and \eqref{est:(g,h)-(e,f)}, the above inequalities allow us to choose $t_0 = t_0 \left(N, \sigma, \alpha_1, \alpha_2, K_1\right)\in (0, t_1)$  small enough that 
$$ 
C_2 T^{\beta}<1,\, \forall\,  T\in(0, t_0]. 
$$ 
With such $t_0>0$ and from \eqref{eq:self_map} and \eqref{est:contraction-map-gamma+chi}, 
we conclude that the operator 
$$
\mathcal{G} :X^{T}_{(g_0,h_0)} \cap B_{K_1} \to X^{T}_{(g_0,h_0)} \cap B_{K_1}
$$ 
is a self-map and a strict contraction for any $T\in(0,t_0]$. 

By applying the Banach Fixed Point Theorem to the set $X^{t_{0}}_{(g_0,h_0)} \cap B_{K_1}$, we conclude that there exists a \emph{unique} fixed point
$(g,h) \in X^{t_{0}}_{(g_0,h_0)} \cap B_{K_1}$, which is also a solution to \eqref{eq:nonlinear-g-h_l}.

\item[\bf (ii).] {\bf Higher-Order Smoothness, i.e.,  
the higher-order smoothness of local solutions to \eqref{eq:nonlinear-g-h_l}}:

The interior smoothness
$g_l,h_l \in C^{\infty}\bigl((0,t_0]\times[0,1]\bigr)$ for all $l$ is obtained by a standard bootstrapping argument for linear parabolic systems, as in the previous section.
This relies on time cut-off functions, compatibility conditions, and repeated applications of Schauder estimates. Since this procedure has already been carried out in detail, we omit the repetition and refer the reader to \cite[Theorem~5.10]{LT26} for further details.

\item[\bf (iii).]{\bf Regularity of $U$}: 
From \eqref{eq:parametrization-chi}, we derive that  
\begin{align*}
U_{l}(t,x)=&g_{l}\left(t, (-1)^{l+1} 
\left(x-x_{2\left[\frac{l}{2}\right]} \right)\right) , 
~~~~~~~~~~~~~~~ x_{l-1} \leq x \leq x_{l}, 
~~ l \in \{1, \ldots, q\},
\\
V_{l}(t,x)=&h_{l}\left(t, (-1)^{l+1} 
\left(x-x_{2\left[\frac{l}{2}\right]} \right)\right) , 
~~~~~~~~~~~~~~~ x_{l-1} \leq x \leq x_{l} , 
~~ l \in \{1, \ldots, q\}.
\end{align*} 
The proof is obtained by applying the result in item (i). 
\end{itemize} 

\end{proof}

\subsection{The Long-Time Existence and Asymptotics }
\label{sec:LTE}

For tangent vector fields $v, v_{1}, \ldots, v_{k}$ on $\mathbf{SU}(N)$, we define  
\begin{equation*}
v_{1}\ast \cdot\cdot\cdot \ast v_{k}=\left\{
\begin{array}{l} 
\langle v_{i_{1}}, v_{i_{2}}\rangle \cdot \cdot \cdot \langle v_{i_{k-1}}, v_{i_{k}}\rangle \text{ \ \ \ \ \ \ \ \ ,\ for }k\text{\ even,} \\
\langle v_{i_{1}}, v_{i_{2}}\rangle \cdot \cdot \cdot \langle v_{i_{k-2}}, v_{i_{k-1}}\rangle \cdot v_{i_{k}}
\text{ , for }k\text{\, odd,}
\end{array}
\right.
\end{equation*}
where $i_{1}, \ldots, i_{k}$ is any permutation of $1, \ldots, k$, $\langle\cdot,\cdot\rangle$ is the inner product, which is defined in \eqref{eq:Matrix_inner_product}.
The notation $P_b^{a,c}(v)$ refers to a finite linear combination of terms
\[
D_x^{\mu_1} v \ast \cdots \ast D_x^{\mu_b} v,
\]
with $\mu_1+\cdots+\mu_b = a$ and $\max\{\mu_1,\ldots,\mu_b\}\le c$, whose coefficients are universal in the sense that they depend only on the bi-invariant Riemannian metric and its derivatives. Moreover, these coefficients and the number of terms in $P_b^{a,c}(v)$ are bounded from above and below by constants depending only on $a$ and $b$.
Note that the following identities follow directly from the definition: 
\begin{equation*}
\left\{
\begin{array}{l} 
D_{x}\left( P_{b_{1}}^{a_{1},c_{1}}\left( v \right) \ast P_{b_{2}}^{a_{2},c_{2}}\left( v \right) \right) 
=D_{x} P_{b_{1}}^{a_{1},c_{1}}\left( v \right) \ast P_{b_{2}}^{a_{2},c_{2}}\left( v \right) 
+P_{b_{1}}^{a_{1},c_{1}}\left( v \right) \ast D_{x} P_{b_{2}}^{a_{2},c_{2}}\left( v \right) \text{,} 
\\ \\
P_{b_{1}}^{a_{1},c_{1}}\left( v \right) \ast P_{b_{2}}^{a_{2},c_{2}}\left( v \right)
=P_{b_{1}+b_{2}}^{a_{1}+a_{2},max\{c_1,c_2\}}\left( v \right) 
\text{, }D_{x}P_{b_{2}}^{a_{2},c_{2}}\left( v \right) 
=P_{b_{2}}^{a_{2}+1,c_{2}+1}\left( v \right) 
\text{.} 
\end{array}
\right. 
\end{equation*} 
The notation $Q_{b}^{a,c}(U_{x})$ is introduced to simplify the summation of polynomials.
Namely,   
\begin{equation*}
Q_{b}^{a, c}(U_{x}) 
:= \sum_{ \substack{  [\![\mu, \nu]\!]\leq [\![a,b]\!] \\\xi \leq c} }   P_{\nu}^{\mu, \xi}(U_{x}) 
=\sum_{\mu=0}^{a}\sum_{\nu=1}^{2a+b-2\mu}\sum_{\xi=0}^{c}\text{ } P_{\nu}^{\mu, \xi}(U_{x})
\text{,} 
\end{equation*}  
where 
$[\![\mu,\nu]\!]:=2\mu+\nu$ 
and $a, c \in \mathbb{N}_{0}$, $b \in \mathbb{N}$. 
For simplicity, the notations $P_b^{a}(U_x)$ and $Q_b^{a}(U_x)$ are also used to denote the corresponding quantities $P_b^{a,a}(U_x)$ and $Q_b^{a,a}(U_x)$, respectively. 
For further details on the notation of the polynomials $P_b^{a}(U_x)$ and $Q_b^{a}(U_x)$, we refer the reader to \cite{LT26}.


\begin{proof}[{\bf The Proof of Theorem \ref{thm:Main_Thm_1}}]

Based on the short-time existence of smooth solutions to
\eqref{U_flow}$\sim$\eqref{Br_boundary-condition-7}, it remains to establish global existence and describe the asymptotic behaviour.
This is achieved by a contradiction argument based on higher-order energy estimates. 
Namely, assume by contradiction that the maximal existence time satisfies 
$t_{\max}<\infty$. We introduce the higher-order energy
\begin{equation}\label{eq:Z_mu}
\mathcal{Z}_{\mu}(t)
:=
\sum_{l=1}^q \|D_t^{\mu-1}\partial_t U_l\|^2
+\frac{1}{\sigma^2}\sum_{l=1}^q \|D_t^{\mu-1}\partial_t V_l\|^2 ,
\end{equation}
and show that $\mathcal{Z}_{\mu}(t)$ remains uniformly bounded on $[t_0,t_{\max})$ for every $\mu\in\mathbb{N}$, which yields a contradiction.

Following the same approach as in \cite{LT26}, specifically applying integration by parts twice together with \cite[Lemma 5.5]{LT26}, we obtain 
\begin{align*}
\dfrac{d}{dt}\frac{1}{2}\sum\limits_{l=1}^q \|  D_t^{\mu-1} \partial_t U_l\|^2_{L^2(I_l)} 
= & - \sum_{l=1}^{q}\int_{I_{l}} | D_{x}^{4\mu+1} U_{l, x} |^2 \text{ }dx +\sum_{l=1}^{q} \int_{I_{l}} Q^{8\mu, 4\mu+1 }_{4}(U_{l,x}) \text{ }dx 
\\ 
& -\sum\limits_{l=1}^q \langle D_t^{\mu} D_x^2 U_{l,x},  D_t^{\mu-1} \partial_t U_l \rangle_{|_{\partial{I}_l}}+\sum\limits_{l=1}^q \langle D_t^{\mu} D_x U_{l,x}, D_x D_t^{\mu-1} \partial_t U_l \rangle_{|_{\partial{I}_l}}.
\end{align*}
Similarly, we have 
\begin{align*}
\dfrac{d}{dt}\frac{1}{2}  \sum\limits_{l=1}^q \|  D_t^{\mu-1} \partial_t V_l\|^2_{L^2(I_l)}
=& -\sum_{l=1}^{q}\int_{I_{l}} | D^{2\mu}_x  V_{l, x}|^2 \text{ }dx +\sum_{l=1}^{q} \int_{I_{l}} Q^{4\mu-2, 2\mu-1 }_{4}(V_{l,x}) \text{ }dx 
\\
& -\sum\limits_{l=1}^q \langle D_t^{\mu}  V_{l,x},  D_t^{\mu-1} \partial_t V_l \rangle_{|_{\partial{I}_l}}.
\end{align*} 
By combining the two identities above together with the boundary conditions \eqref{Br_clampend-pass-point}$\sim$\eqref{Br_boundary-condition-7}, we have 
\begin{align*}
\dfrac{d}{dt}\mathcal{Z}_\mu(t)=& - \sum_{l=1}^{q}\int_{I_{l}} | D_{x}^{4\mu+1} U_{l, x} |^2 \text{ }dx -\frac{1}{\sigma^2}\sum_{l=1}^{q}\int_{I_{l}} | D^{2\mu}_x  V_{l, x}|^2 
\, dx
\text{.}
\end{align*} 
From these identities, and by applying \cite[Lemma 5.13]{LT26} together with Young inequaltiy \cite[Lemma A.1]{DLP19} with $\varepsilon=\frac{1}{2}$, we find
\begin{align*} 
\nonumber
\frac{d}{dt}\mathcal{Z}_{\mu}(t)
+&\mathcal{Z}_{\mu}(t)
+2\sum_{l=1}^{q}\int_{I_{l}}| D_{x}^{4\mu+1} U_{l,x}|^2 \text{ }dx
+ \frac{2}{\sigma^2}\sum_{l=1}^{q}\int_{I_{l}} | D_{x}^{2\mu}V_{l, x} |^2 \text{ }dx\\ \nonumber
=&\sum_{l=1}^{q} \int_{I_{l}} Q^{8\mu, 4\mu+1 }_{4}(U_{l,x}) \text{ }dx +\sum_{l=1}^{q} \int_{I_{l}} Q^{4\mu-2, 2\mu-1}_{2}(V_{l,x}) \text{ }dx 
\\
& \leq  \frac{1}{2} 
\left( 
\sum_{l=1}^{q}\int_{I_{l}} | D_{x}^{4\mu+1} U_{l, x} |^2 \text{ }dx+
 \frac{1}{\sigma^2}\sum_{l=1}^{q}\int_{I_{l}} | D_{x}^{2\mu} V_{l, x} |^2 \text{ }dx \right) 
+C\left(N, p_0, \sigma, \mathcal{F}_{\sigma}(t_0)\right), 
\end{align*} 
where $\mathcal{F}_{\sigma}(t_0)=\mathcal{F}_{\sigma}\left([U(t_0,\cdot),V_1(t_0,\cdot), \ldots, V_{q}(t_0,\cdot)]\right)$, which implies that
\begin{align}
\label{inequality-D^{j-1}_t-partial_{t}gamma-fitting}
\frac{d}{dt}\mathcal{Z}_{\mu}(t)&+\mathcal{Z}_{\mu}(t) \leq C\left(N, p_0, \sigma, \mathcal{F}_{\sigma}(t_0) \right). 
\end{align} 
By applying Gr\"{o}nwall inequality to \eqref{inequality-D^{j-1}_t-partial_{t}gamma-fitting}, we obtain 
\begin{align}
&\mathcal{Z}_{\mu}(t)
\leq C\left(N, p_0, \mu, \mathcal{F}_{\sigma}(t_0), 
\mathcal{Z}_{\mu}(t_0) \right).
\label{eq:mainesti} 
\end{align} 
Then, from applying \cite[Lemma 5.4]{LT26} and \cite[Lemma $A.1$]{DLP19} to \eqref{eq:mainesti}, we obtain 
\begin{equation}\label{unif-bdd:D-2kj+k-3-derivative} 
\begin{cases}
\sum\limits_{l=1}^{q} \| D_{x}^{4\mu-1} U_{l,x} \|^2_{L^2(I_{l})} 
\leq  C\left(N, \mu, \mathcal{F}_{\sigma}(t_0), 
\mathcal{Z}_{\mu}(t_0) \right),
\\
\sum\limits_{l=1}^{q} \| D_{x}^{2\mu-1} V_{l,x} \|^2_{L^2(I_{l})} 
\leq C\left(N, p_0, \mu, \mathcal{F}_{\sigma}(t_0), 
\mathcal{Z}_{\mu}(t_0) \right).
\end{cases}
\end{equation}
Thus, the required uniform estimates for \eqref{eq:Z_mu} are derived for any $\mu\in\mathbb{N}$.

The asymptotic behaviour of global solutions follows by arguments analogous
to those in \cite{LT26}. The uniform bounds on higher-order derivatives of
$U_l$ and $V_l$, see \eqref{unif-bdd:D-2kj+k-3-derivative}, yield the existence
of a sequence $t_j\to\infty$ and a limit $U_\infty$ such that
$U(t_j,\cdot)\to U_\infty$.
Moreover, the uniform estimates
\eqref{inequality-D^{j-1}_t-partial_{t}gamma-fitting} and \eqref{eq:mainesti}
imply $\mathcal{Z}_1\in L^1([0,\infty))$ and $\mathcal{Z}_1(t)\to0$ as
$t\to\infty$. 
This completes the proof; for further details, we refer the reader to \cite{LT26}. 

\end{proof}

{\bf Acknowledgement.} 
This project is supported by the following grants:
C.-C. L. received support from the research grant 
(NSTC 113-2115-M-003-014-MY3) of the National Science and Technology Council (NSTC) of Taiwan, and the National Center for Theoretical Sciences (NCTS), Taiwan. 
T. D. T expresses his gratitude to the scholarship fund for postdoctoral fellows of the VNU University of Science for their support. He would also like to thank the Vietnam Institute for Advanced Study in Mathematics (VIASM) for its hospitality and for the excellent working conditions.

{\bf Data availability.} 
Data sharing is not applicable to this article, as no datasets were generated or analyzed during the current study.

{\bf Declarations.} 
The authors declare no conflicts of interest related to this publication.


\begin{thebibliography}{99}  


\bibitem{BCC21}
A. Bloch, M. Camarinha, and L. Colombo, 
Variational point-obstacle avoidance on Riemannian manifolds. 
\textit{Math. Control Signals Systems}, 33 (2021), no. 1, 109--121. 

\bibitem{BHM12} D. C. Brody, D. D. Holm, and D. M. Meier, 
Quantum splines, \textit{Phys. Rev. Lett. }, 109, 100501 (2012). 


\bibitem{CH24} A. Cabrera and R. L. Hatton, 
Optimal control of robotic systems and biased Riemannian splines.  
\textit{ESAIM Control Optim. Calc. Var.} 30 (2024), no. 36.


\bibitem{CSC22} M. Camarinha, F. Silva Leite, and P. Crouch,  Riemannian cubics close to geodesics at the boundaries.  \textit{J. Geom. Mech.} 14 (2022), no. 4, 545--558. 


\bibitem{DLP19} 
A. Dall'Acqua, C.-C. Lin, and P. Pozzi, Flow of elastic networks: long time existence results. 
\textit{Geom. Flows.} 4 (2019), no. 1, 83--136. 

\bibitem{DLP21} 
A. Dall'Acqua, C.-C. Lin, and P. Pozzi, Elastic flow of networks: short-time existence result. 
\textit{J. Evol. Equ.} 21 (2021), no. 2, 1299--1344. 


\bibitem{GHMRV12-1} F. Gay-Balmaz, D. D. Holm, D. M. Meier, T. S. Ratiu, and F.-X. Vialard, Invariant higher-order variational problems, \textit{Commun. Math. Phys.} 309 (2012), 413--458.
 
 
\bibitem{GHMRV12-2}F. Gay-Balmaz, D. D. Holm, D. M. Meier, T. S. Ratiu, and F.-X. Vialard, Invariant higher-order variational problems II, \textit{J. Nonlinear Sci.} 22 (2012), 553--597.

\bibitem{GGP03} R. Giamb\`{o}, F. Giannoni, and P. Piccione, Optimal Control on Riemannian Manifolds by Interpolation, \textit{Math. Control Signals Syst.} 16 (2003), no. 4,  278--296. 

\bibitem{HRW19} B. Heeren, M. Rumpf, and B. Wirth, Variational time discretization of Riemannian splines. \textit{IMA J. Numer. Anal.} 39 (2019), no. 1, 61--104. 

\bibitem{LT26} C.-C. Lin and D. T. Tran, 
Higher-order Riemannian spline interpolation problems: a unified approach via gradient flows, \textit{Math. Ann.} 394 (2026), no. 61.  

\bibitem{MTM17} L. Machado, T. Teresa, and M. Monteiro, A numerical optimization approach to generate smoothing spherical splines. \textit{J. Geom. Phys.} 111 (2017), 71--81. 

\bibitem{Milnor63} 
J. Milnor, Morse theory, 
\textit{Ann. of Math. Stud.}, No. 51, 
Princeton University Press, Princeton, NJ, 1963.

\bibitem{N03} L. Noakes, Null cubics and Lie quadratics. \textit{J. Math. Phys.} 44 (2003), no. 3, 1436--1448. 


\bibitem{NHP89} L. Noakes, G. Heinzinger, and B. Paden, Cubic splines on curved spaces. \textit{IMA J. Math. Control Inform.} 6 (1989), no. 4, 465--473. 


\bibitem{Nomizu54}
K. Nomizu, 
Invariant affine connections on homogeneous spaces. 
\textit{Amer. J. Math.}, 76 (1954), 33--65. 



\bibitem{Solonnikov65} V. A. Solonnikov, Boundary Value Problems of Mathematical Physics. III. 
\textit{Proceedings of the Steklov institute of Mathematics (1965), Amer. Math. Soc., Providence, R. I.}, No. 83, 1967. 
\end{thebibliography}
\end{document}